\documentclass[a4paper, reqno]{amsart}

\usepackage[mathscr]{eucal}

\usepackage{amsmath,amssymb,amsfonts,amsthm}

\newtheorem{theorem}{Theorem}[section]

\newtheorem{lemma}[theorem]{Lemma}

\newcommand{\N}{\mathbb N}
\newcommand{\Z}{\mathbb Z}
\newcommand{\R}{\mathbb R}

\newcommand{\F}{\mathbb F}

\newcommand{\HH}{\mathsf H}

\newcommand{\Fc}{\mathcal F}
\newcommand{\vp}{\mathsf v}

\newcommand{\dd}{\mathsf d}

\newcommand{\ord}{\text{\rm ord}}

\newcommand{\supp}{\text{\rm supp}}

\newcommand{\Ker}{\text{\rm Ker}}

\newcommand{\la}{\langle}
\newcommand{\ra}{\rangle}

\renewcommand{\t}{\, | \,}

\newcommand{\be}{\begin{equation}}
\newcommand{\ee}{\end{equation}}

\newcommand{\bnml}{\begin{multline}}
\newcommand{\enml}{\end{multline}}

\newcommand{\ber}{\begin{eqnarray}}
\newcommand{\eer}{\end{eqnarray}}
\newcommand{\nn}{\nonumber}
\newcommand{\Sum}[2]{\underset{#1}{\overset{#2}{\sum}}}
\newcommand{\Summ}[1]{\underset{#1}{\sum}}

\newcommand{\und}{\;\mbox{ and }\;}

\newcommand{\bdot}{\boldsymbol{\cdot}}

\makeatletter
\newcommand{\superimpose}[2]{%
  {\ooalign{$#1\@firstoftwo#2$\cr\hfil$#1\@secondoftwo#2$\hfil\cr}}}
\makeatother

\newcommand{\bulletprod}[1]{\underset{#1}{\bullet}}

\begin{document}

\title[The Large Davenport Constant I]{The Large Davenport Constant I: \\ Groups with a Cyclic, Index $2$ Subgroup}
\thanks{This work was supported by
the {\it Austrian Science Fund FWF} (Project No. P21576-N18)}

\author{Alfred Geroldinger and David J. Grynkiewicz}

\address{Institut f\"ur Mathematik und Wissenschaftliches Rechnen \\
Karl-Franzens-Universit\"at Graz \\
Heinrichstra\ss e 36\\
8010 Graz, Austria} \email{alfred.geroldinger@uni-graz.at, diambri@hotmail.com}



\begin{abstract}
Let $G$ be a finite group written multiplicatively. By a sequence over $G$, we mean a finite sequence of terms from $G$ which is unordered, repetition of terms allowed, and we say that it is a  product-one sequence if its terms can be ordered so that their product is the identity element of $G$.
The {\it small Davenport constant} $\mathsf d (G)$ is the maximal integer $\ell$ such that there is a sequence over $G$ of length $\ell$ which has  no nontrivial, product-one subsequence. The {\it large Davenport constant} $\mathsf D (G)$ is the maximal length of a minimal product-one sequence---this is a product-one sequence which cannot be factored  into two nontrivial, product-one subsequences.
It is easily observed that $\dd(G)+1\leq \mathsf D(G)$, and if $G$ is abelian, then equality holds. However, for non-abelian groups, these constants can differ significantly.
Now suppose $G$ has a cyclic, index $2$ subgroup.
Then an old result of Olson and White (dating back to  1977) implies that $\dd(G)=\frac12|G|$ if  $G$ is non-cyclic, and $\dd(G)=|G|-1$ if $G$ is cyclic. In this paper, we determine the large Davenport constant of such groups, showing that $\mathsf D(G)=\dd(G)+|G'|$, where $G'=[G,G]\leq G$ is the commutator subgroup of $G$.
\end{abstract}

\maketitle


\section{Introduction and Main Result}

Let $G$ be a multiplicatively written, finite group. A sequence $S$ over $G$ means a finite sequence of terms from $G$ which is unordered, repetition of terms allowed. We say that $S$ is a product-one sequence if its terms can be ordered so that their product equals $1$, the identity element of the group. The {\it small Davenport constant} $\mathsf d (G)$ is the maximal integer $\ell$ such that there is a sequence over $G$ of length $\ell$ which has  no nontrivial, product-one subsequence. The {\it large Davenport constant} $\mathsf D (G)$ is the maximal length of a minimal product-one sequence---this is a product-one sequence which cannot be partitioned into two nontrivial, product-one subsequences. A simple argument shows that $\mathsf d (G)+1 \le \mathsf D (G) \le |G|$.

The problem of finding the precise value  of the Davenport constant and what is now known as the  Erd{\H o}s--Ginzburg--Ziv Theorem became the starting points of Zero-Sum Theory. Since that time (dating back to the early 1960s), it has developed into a flourishing branch of Additive and Combinatorial Number Theory. We briefly discuss some of the motivation for these problems. For more detailed information, we defer to   the surveys \cite{Ca96b, Ga-Ge06b, Ge09a} or the monographs \cite{Ge-HK06a, Gr13a}.
Apart from abelian groups, the Davenport constant has also been studied for finite abelian (non-cancellative) semigroups (see \cite{Wa-Ga08a}, \cite[Proposition 2.8.13]{Ge-HK06a}).

Although the main focus of Zero-Sum Theory has been on abelian groups, research was never restricted to the abelian setting alone. To provide one example apart from the Davenport constant, let $\mathsf E (G)$ denote the smallest integer $\ell$ guaranteeing that every sequence $S$ over $G$ of length  $|S| \ge \ell$ has a product-one subsequence of length $|G|$. Motivated by the classical  Erd{\H o}s--Ginzburg--Ziv Theorem, the study of $\mathsf E (G)$ has attracted much attention for non-abelian groups \cite{Zh-Ga05a, Ba07b, Ga-Lu08a, Ga-Li10b}. In all cases studied so far (abelian and non-abelian), it has turned out that $\mathsf E (G) = |G| + \mathsf d (G)$.

If $G$ is a finite abelian group,
then an easy observation shows that $\mathsf d (G)+1 = \mathsf D (G)$. Here there is no difference between the combinatorially defined small davenport constant $\dd(G)+1$ and the  monoid theoretic large Davenport Constant $\mathsf D(G)$. In this classical setting,
the Davenport constant was first introduced  by Rogers \cite{Ro63} (though Davenport became more famous for promoting it) who pointed out a connection between $\mathsf D (G)$ and irreducible elements in a ring of algebraic integers with ideal class group isomorphic to $G$ (see Section \ref{sec-notation}). His observation was deepened by Narkiewicz \cite{Na79} whose paper was the first step in the creation of a strong bridge between the arithmetic of Krull monoids and Additive Combinatorics (via the associated monoid of zero-sum sequences over the class group; see \cite{Ge09a} for a survey).

The first attempts to study a Davenport constant in a non-abelian setting were carried out by Olson and White \cite{Ol-Wh77}, who defined the small Davenport constant of a non-abelian group and gave a general upper bound that was shown to be tight for groups having a cyclic, index $2$ subgroup; see Theorem \ref{thm-olsonwhite-upper}. However, the definition of the small Davenport constant is not as fully satisfying in this setting.
The first reason for this is simple: there is no monoid factorization interpretation of the small Davenport constant over a non-abelian group. The second reason for this regards Invariant Theory and the Noether constant.

Let $G$ be a finite group, let $\F$ be a field whose characteristic does not divide $|G|$, and let $\beta (G)$ denote the Noether constant, which is defined as the maximal degree of an invariant polynomial in a minimal generating set of the invariant ring $\F[V]^G$. When $G$ is abelian, we have $\mathsf d(G)+1=\mathsf D(G)$. However, when $G$ is non-abelian, there are examples where $\dd(G)+1<\beta(G)$, meaning the small Davenport constant cannot be used for bounding $\beta(G)$ from above. Attempting to rectify this problem, we have defined the large Davenport constant simply by taking the natural monoid theoretic definition and extending it to non-abelian groups.  The natural conjecture is that $\mathsf d (G)+1 \le \beta (G) \le \mathsf D (G)$ might hold for   groups in general.  By the results of the present paper, this conjecture holds for all groups having a cyclic, index $2$ subgroup. For more on invariant theory and the Noether constant, we refer the reader to
the monographs \cite{Ne-Sm02a, Ne07a} or to more recent work \cite{Ne07b, Ge-Ne13b}.

Our main result is the following theorem, in which we parallel the early result of Olson and White \cite{Ol-Wh77} that determined the small Davenport constant of a finite group having a cyclic, index $2$ subgroup, by instead determining the large Davenport constant for all such groups. Theorem \ref{thm-index2} covers dihedral groups, semi-dihedral groups, and generalized quaternion or dicyclic groups, as well as many more. Building upon the results of this paper, we will give more general upper bounds for $\mathsf D(G)$  in a sequel \cite{Gr14a}.

\begin{theorem}\label{thm-index2}
Let $G$ be a finite group having a cyclic, index $2$ subgroup.  Then
\[
 \mathsf D(G)=\mathsf d(G)+|G'|\quad\und\quad\mathsf d(G)=\left\{
                                                            \begin{array}{ll}
                                                              |G|-1 & \hbox{if $G$ is cyclic} \\
                                                              \frac12|G| & \hbox{if $G$ is non-cyclic,}
                                                            \end{array}
                                                          \right.
\]where $G'=[G,G]\leq G$ is the commutator subgroup of $G$.
\end{theorem}

The paper is divided as follows. In Section \ref{sec-notation}, we introduce and adapt notation used for sequences and sumsets over abelian groups  and prove several basic facts. In Section \ref{sec-upper-bounds}, we give some general upper bounds that can be used in conjunction with inductive arguments. Section \ref{sec-add} deals entirely with classical results for abelian groups, needed for later proofs, and the proof of one axillary lemma needed for handling dicyclic groups. The main bulk of the proof of Theorem \ref{thm-index2} is then carried out in Section \ref{sec-index2}, beginning with an overview  of  the possible isomorphism classes of groups having a cyclic, index $2$ subgroup.


\section{Notation and Preliminaries}\label{sec-notation}

All intervals will be discrete, so for real numbers $a,\, b \in \R$, we
set $[a, b] = \{ x \in \Z :\; a \le x \le b \}$.
If $A$ and $B$ are sets, then whenever addition or multiplication between elements of $A$ and $B$ is allowed, we define their sumset and product-set as
 $$A+B=\{a+b:\; a\in A,\,b\in B\}\quad\und\quad AB=\{ab:\;a\in A,\, b\in B\}.$$
Of course, we use the abbreviations $A+g=\{a+g:\; a\in A\}$, \  $Ag=\{ag:\;a\in A\}$ and $gB=\{gb:\;b\in B\}$ when dealing with a single element $g$ for which the respective addition or multiplication is defined.

In our main applications, all groups will be finite, but we will encounter groups written both additively and multiplicatively, reserving addition only for cases where it is a commutative operation. For the moment, assume that $G$ is a group written multiplicatively except when otherwise noted.

If $A\subset G$ is a nonempty subset, then we use $\la A\ra\leq G$ to denote the subgroup generated by $A$
 and use $\mathsf H(A):=\{g\in G:\; gA=A\}$ to denote the left {\it stabilizer} of $A$. Then $\mathsf H(A)\leq G$ is a subgroup, and  $A$ is a union of right $\mathsf H(A)$-cosets; moreover,  $\mathsf H(A) \leq G$ is the unique maximal subgroup $H$ for which $A$ is a union of right $H$-cosets.
Of course, if $G$ is abelian, then we do not need to differentiate between left and right stabilizers and simply speak of the  stabilizer of  $A$, and when $G$ is written additively, we have $\mathsf H(A)=\{g\in G:\;g+A=A\}$. For $n\geq 1$, we let $C_n$ denote a cyclic group of order $n$.

Given a normal subgroup $H \lhd G$, we let $$\phi_H \colon G\rightarrow G/H$$ denote the canonical homomorphism. The index of a subgroup $H\leq G$ is denoted $|G:H|$. When $G$ is finite, $|G:H|=|G|/|H|$. We use  standard notation for the following important subgroups:  \begin{align*}
&\mathsf Z(G)=\{g\in G:\; gx=xg \quad\mbox{ for all $x\in G$}\}\lhd G\quad\mbox{ is the \emph{center} of $G$},\\
&[x,y]=x^{-1}y^{-1}xy\in G\quad\mbox{ is the \emph{commutator} of the elements $x,\,y\in G$},\\
& G'=[G,G]=\la [x,y]:\; x,\,y\in G\ra\lhd G\quad\mbox{ is the \emph{commutator subgroup} of $G$}, \quad\und \\
& \mathsf C_G(A)= \mathsf C_G(\la A\ra)=\{g\in G:\; ga=ag\quad\mbox{ for all $a\in A$}\}\leq G \quad\mbox{ is the
\emph{centralizer} of $A\subset G$}.
\end{align*}

For a set $P$, we denote by $\mathcal F (P)$ the \ {\it free
abelian monoid} \ with basis $P$.
Then  every $a \in \mathcal F
(P)$ has a unique representation in the form
\[
a = p_1 \bdot \ldots \bdot p_{\ell} = \prod_{p \in P} p^{\mathsf v_p(a), } \quad \text{where} \ p_1, \ldots , p_{\ell} \in P , \quad
\mathsf v_p(a) \in \N_0 \ \text{ and } \ \mathsf v_p(a) = 0 \ \text{
for almost all } \ p \in P \,,
\]
and we use all notation from elementary divisibility theory. In particular, $\mathsf v_p (a)$ is the $p$-{\it adic valuation} of $a$, \
$\supp (a) = \{p \in P :\; \mathsf v_p (a) > 0\} \subset P$ is the {\it support} of $a$, \
$|a|= \ell = \sum_{p \in P}\mathsf v_p(a)$ is the \emph{length} of
$a$, and $\mathsf h (a) = \max \{ \mathsf v_p (a) :\; p \in P \}$.

\subsection*{Sequences Over Groups}
These are our main objects of study. As it is traditional in Combinatorial Number Theory, by  a {\it sequence} over a group $G$ we mean a finite, unordered sequence where the repetition of elements is allowed. We view sequences over $G$ as elements of the free abelian monoid $\mathcal F (G)$ (this point of view provides many technical advantages and was pushed forward by applications of Zero-Sum Theory in more algebraic fields, such as  Multiplicative Ideal Theory and  Factorization Theory; see  the monographs \cite{Ge-HK06a, Gr13a} or the surveys \cite{Ge-HK06b, HK08a, Sc09b, Ba-Ch11a}). So we freely use all notation from free abelian monoids for sequences, though for reason explained in the next paragraph, we denote multiplication in $\Fc(G)$ by the boldsymbol $\bdot$ rather than by juxtaposition and use brackets for all exponentiation in $\Fc(G)$. In particular, a sequence $S\in \Fc(G)$ has the form
\be\label{snot}S=g_1\bdot\ldots\bdot g_\ell=\bulletprod{i\in [1,\ell]} g_i\in \Fc(G)\ee with the $g_i\in G$ the terms of $S$.
The identity $1_{\Fc(G)} \in \mathcal F (G)$ is called the {\it empty} or {\it trivial } sequence, which is simply the sequence having no terms. For $g \in G$, \ \begin{align*}&\mathsf v_g (S)=|\{i\in [1,\ell]:\; g_i=g\}|\quad\mbox{ denotes the {\it multiplicity} of $g$ in $S$,}\\ &\mathsf h(S)=\max\{\vp_g(S):\; g\in G\}\quad \mbox{ denotes the {\it maximum multiplicity } of a term of $S$,}\end{align*}
and $T\mid S$ denotes that $T$ is a {\it subsequence} of $S$. Of course, for $T \in \mathcal F (G)$, we have $T \t S$ if and only if $\mathsf v_g (T) \le \mathsf v_g (S)$ for all $g \in G$, and in such case,  $T^{[-1]}\bdot S$ or $S\bdot T^{[-1]}$ denotes the subsequence of $S$ obtained by removing the terms of $T$ from $S$, i.e.,  $\mathsf v_g (T^{[-1]}\bdot S) = \mathsf v_g (S) - \mathsf v_g (T)$ for all $g \in G$.

In order to distinguish between the group operation in $G$ and the sequence operation in $\mathcal F (G)$, we use the boldsymbol $\bdot$ for the operation in $\Fc(G)$,  so $\mathcal F (G) = (\mathcal F (G), \bdot)$ (which coincides with the convention in the monographs \cite{Ge-HK06a, Gr13a}) and only denote multiplication in $G$ by juxtaposition of elements. In particular,  if $S_1, S_2 \in \mathcal F (G)$ and $g_1, g_2 \in G$, then $S_1 \bdot S_2 \in \mathcal F (G)$ has length $|S_1|+|S_2|$, \ $S_1 \bdot g_1 \in \mathcal F (G)$ has length $|S_1|+1$, \ $g_1g_2 \in G$ is an element of $G$, but $g_1 \bdot g_2 \in \mathcal F (G)$ is a sequence of length $2$.
In order to avoid confusion between exponentiation of the group operation in $G$ and exponentiation of the sequence operation $\bdot$ in $\Fc(G)$, we use brackets to denote exponentiation in $\mathcal F (G)$:
 $$g^{[k]}=\underset{k}{\underbrace{g\bdot\ldots\bdot g}}\in \Fc(G)\quad\und\quad T^{[k]}=\underset{k}{\underbrace{T\bdot\ldots\bdot T}}\in \Fc(G),$$ for $g \in G$, \ $T\in \Fc(G)$ and $k \in \mathbb N_0$.
When $T^{[k]}\mid S$, we extend exponentiation to include negative exponents by setting $S\bdot T^{[-k]}=S\bdot (T^{[k]})^{[-1]}\in \Fc(G)$. In particular,
 if $S \in \mathcal F (G)$, \ $g \in G$ and $k \in \mathbb Z$ with $k \ge -\mathsf v_g (S)$, then $S \bdot g^{[k]} \in \mathcal F (G)$ has length $|S| + k$.

Let $S\in \Fc(G)$ be a sequence notated as in \eqref{snot}.
When  $G$ is written multiplicatively,  we use
\[
\pi (S) = \{ g_{\tau (1)} \ldots  g_{\tau (\ell)} \in G :\; \tau\mbox{ a permutation of $[1,\ell]$} \} \subset G
\]
to denote the {\it set of products} of $S$. In view of the basic properties of the commutator subgroup $G'=[G,G]\leq G$, it is readily seen that $$\pi(S)\quad\mbox{ is contained in a $G'$-coset}.$$
Note that $|S| = 0$ if and only if $S$ is trivial, and in this case we use the convention that $\pi (S) = \{1\}$.
When $G$ is written additively with commutative operation, we likewise let $$\sigma(S)=g_1+\ldots+g_\ell\in G$$ denote the \emph{sum} of $S$. More generally,
for any integer $n\geq 0$, the $n$-{\it sums} and \emph{$n$-products} of $S$ are respectfully denoted by
\[
\Sigma_n (S) = \{\sigma(T):\; T \t S\;\und\; |T|=n\}  \subset G \quad\und\quad \Pi_n (S) = \underset{|T|=n}{\bigcup_{T \t S}} {\pi}(T)  \subset G \,,
\]
and the {\it subsequence sums } and \emph{subsequence products} of $S$ are respectively denoted by
\[
\Sigma (S) = \bigcup_{n\geq 1}{\Sigma_n}(T)  \subset G \quad\und\quad \Pi (S) = \bigcup_{n\geq 1}{\Pi_n}(T) \subset G \,.
\]

The sequence $S$ is called
\begin{itemize}
\item a {\it product-one sequence} if $1 \in \pi (S)$,

\item {\it product-one free} if $1 \notin \pi (S)$.
\end{itemize}
Zero-sum and zero-sum free sequences are analogously defined when $G$ is written additively using $\sigma$ in place of $\pi$ and $0$ in place of $1$.
Every map of groups $\varphi \colon G \to H$ extends to  a monoid homomorphism $\varphi \colon \mathcal F (G) \to \mathcal F (H)$ by setting
\[
\varphi (S) =  \varphi (g_1) \bdot \ldots \bdot  \varphi (g_\ell)\in \Fc(H) \,.
\]
If $\varphi$ is a group homomorphism, then $\varphi (S)$ is a product-one sequence if and only if $\pi (S) \cap \Ker ( \varphi) \ne \emptyset$.

We use
\[
\mathcal B (G) = \{ S \in \mathcal F (G) :\; 1 \in  \pi (S) \}
\] to denote
the set of all product-one sequences. Clearly, $\mathcal B (G) \subset \mathcal F (G)$ is a submonoid, hence a commutative, cancellative semigroup with unit element, and we denote by
$\mathcal A (G) = \mathcal A  \big( \mathcal B (G) \big)$ the set of atoms (irreducible elements) of $\mathcal B (G)$. In other words, $\mathcal A (G)$ consists of the minimal product-one sequences, which are the nontrivial, product-one sequences that cannot be factored into two nontrivial, product-one subsequences. We call
\[
\mathsf D (G) = \sup \{ |S| :\; S \in \mathcal A (G) \} \in \mathbb N \cup \{\infty\}
\]
the {\it large Davenport constant} of $G$ and
\[
\mathsf d (G) = \sup \{ |S| :\; S \in \mathcal F (G) \ \text{is product-one free} \} \in \mathbb N_0 \cup \{\infty\}
\]
the {\it small Davenport constant} of $G$.

Suppose that  $G$ is abelian. Then $\mathcal B (G)$ is a Krull monoid (for more on Krull monoids, see \cite{HK98, Fa06a, F-H-K-W06}). More precisely (apart from the trivial case  $|G|=2$), the monoid $\mathcal B (G)$
is (up to isomorphism) the  unique reduced Krull monoid  with class group $G$ in which every class contains exactly one prime divisor. When studying the arithmetic of general Krull monoids $H$ (e.g., of integrally closed, noetherian domains) with class group $G$,  many questions can be reduced to the associated monoid $\mathcal B (G)$ of zero-sum sequences over the class group  \cite[Section 3.4]{Ge-HK06a}. For instance, the large Davenport constant $\mathsf D (G)$ is the supremum over all $k$ for which there exists an atom $u \in H$ which is a product of $k$ prime divisors \cite[Theorem 5.1.5]{Ge-HK06a}. For rings of integers in  algebraic number fields (which are Krull monoids), this was first observed by Rogers in 1963 \cite{Ro63}. Thus, from the very beginning up to the latest applications, it has always been the large Davenport constant which has been at the center of interest. In the abelian case,
a simple argument  (see Lemma \ref{dav-up-lemma}) shows that $\dd(G)+1=\mathsf D(G)$. Thus the small Davenport constant is a sufficient tool
to study the large Davenport constant for abelian groups. For general groups, we only  have the inequality $\mathsf d (G) + 1 \le \mathsf D (G)$,  and hence the study of the large Davenport constant requires additional efforts.

\subsection*{Ordered Sequences Over Groups} These are an important tool used to study (unordered) sequences over non-abelian groups.
Indeed, it is quite useful to have related notation for sequences in which the order of terms matters. Thus we let  $\mathcal F^* (G) = (\mathcal F^* (G), \bdot)$ denote the free (non-abelian) monoid with basis $G$, whose elements will be called the {\it ordered sequences} over $G$. In other terminology, $\mathcal F^* (G)$ is the semigroup of words on the alphabet $G$, and the elements are called words or strings.

Taking an ordered sequence in $\mathcal F^* ( G)$ and considering all possible permutations of its terms gives rise to a natural equivalence class in $\mathcal F^*( G)$, yielding  a natural map $$[\bdot]: \Fc^*( G) \rightarrow \Fc(G)$$ given by abelianizing the sequence product in $\Fc^* (G)$.
An ordered sequence $S^*\in \Fc^*(G)$ with $[S^*]=S$ is called an \emph{ordering} of the sequence $S\in \Fc(G)$.

All notation and conventions for sequences extend naturally to ordered sequences. In particular,  every map of groups $\varphi \colon G \to H$ extends uniquely to a monoid homomorphism $\varphi \colon \mathcal F^* (G) \to \mathcal F^* (H)$ and, for an ordered sequence $S^* \in \mathcal F^* (G)$ with $S=[S^*]$, we set $\mathsf h (S^*) = \mathsf h (S)$, \ $\supp (S^*) = \supp (S)$, \ $|S^*| = |S|$, and $\mathsf v_g (S^*) = \mathsf v_g (S)$ for every $g \in G$.
Let
$$S^* =  g_1  \bdot \ldots \bdot  g_\ell \in \mathcal F^* ( G)$$ be an ordered sequence.
For every subset $I \subset [1, \ell]$, we set
 \be\label{ord-subseq-notation}  S^*(I)=  \bulletprod{i\in I} g_i \in \mathcal F^* ( G),\ee
where the product is taken in the natural order given by $I\subset \Z$, and every sequence of such a form  in $\Fc^*(G)$ is called an {\it (ordered) subsequence} of $S^*$. We use the abbreviation $$S^*(x,y)= S^*([x,y])$$ for integers $x,\,y \in \Z$. If $I = \emptyset$, then $S^*(I) = 1_{\mathcal F^* (G)}$ is the identity of $\mathcal F^* (G)$ (in other words, the empty ordered sequence), and
if $T^*= S^*(I)$ with $I \subset [1, \ell]$ an interval, then we say that  $T^*\in \Fc^*(G)$  is a subsequence of {\it consecutive  terms}, or simply a {\it consecutive  subsequence}, and we indicate this by writing  $T^*\mid S^*$.
If $i\in [1,|S^*|]$, then  $$S^*(i)= S^*([i,i]) \in G\quad\mbox{ denotes the $i$-th term of $S^*$}.$$ Let $\pi \colon \mathcal F^* (G) \to G$ denote the unique homomorphism that maps an ordered sequence onto its product in $G$, so $$\pi(S^*)=\prod_{i=1}^\ell g_i\in G.$$
If $\pi(S^*)=1$, then $S^*$ is called a {\it product-one ordered sequence}.

%

By a {\it factorization} of $S^*\in \Fc^*(G)$ of length $r$, we mean an $r$-tuple $(S^*_1, \ldots, S^*_r)$ of nontrivial, consecutive  subsequences $S^*_i\mid S^*$ such that $S^*=S^*_1\bdot \ldots\bdot S^*_r$. Informally speaking, we may refer to $S^*=S^*_1\bdot \ldots\bdot S^*_r$ as a factorization of $S^*$ as well.
Then, for each $i\in [1,r]$, we have  $S^*_i=S^*(I_i)$ for some $I_i\subset [1,|S|]$ such that
$$\bigcup_{i=1}^rI_i=[1,|S|]\quad\und\quad \max I_j=\min I_{j+1}-1\;\mbox{ for $j\in [1,r-1]$}.$$
Given such a factorization of $S^*$, we can define a new ordered sequence $$T^*= \pi(S^*_1)\bdot\ldots\bdot  \pi(S^*_r)\in \Fc^*(G),$$ so $T^*$ is obtained from $S^*$ by replacing consecutive subsequences with the product of their terms. It is then readily noted that $$\pi(T^*)=\pi(S^*) \quad\und\quad \pi([T^*])\subset \pi([S^*]).$$ Moreover, if $[S^*]\in \mathcal A(G)$ was an atom, then $[T^*]\in \mathcal A(G)$ must remain an atom.

\subsection*{Basic Lemmas Regarding Sequences}

We now prove several basic lemmas and observations that will be needed repeatedly in the paper.
The first two  are rather straightforward but frequently needed.

\begin{lemma}\label{lemma-no-consecutive-prod-1}
Let $G$ be a  group and let $U^* \in \Fc^*(G)$ be an ordered sequence with $\pi(U^*)=1$ and $[U^*]\in \mathcal A(G)$ an atom. Then there are no consecutive, product-one subsequences of $U^*$ that are proper and nontrivial.
\end{lemma}

\begin{proof}Observe that removing a consecutive, product-one subsequence from an ordered sequence does not affect its product. Thus,
if the product-one ordered sequence $U^*$ had a consecutive, product-one subsequence that was proper and nontrivial, say $U^*(I)$ with $I\subset [1,|U^*|]$ an interval, then $[U^*]=[U^*(I)]\bdot [U^*([1,|U^*|]\setminus I)]$ would be a factorization of $[U^*]$ into two nontrivial, product-one subsequences, contradicting that $[U^*]\in \mathcal A(G)$ is an atom.
\end{proof}

\begin{lemma}\label{lemma-G'} Let $G$ be group with $G'=[G,G]\leq G$ its commutator subgroup, and
let $S\in \Fc(G)$ be a product-one sequence. If $T\mid S$ is a subsequence with $\pi(T)\subset G'$, then $\pi(T^{[-1]}\bdot S)\subset G'$. In particular, if $T\mid S$ is a product-one subsequence, then $\pi(T^{[-1]}\bdot S)\subset G'$.
\end{lemma}

\begin{proof}
As remarked earlier in the section, we know that every sequence $R\in \Fc(G)$ has $\pi(R)$ contained in a $G'$-coset. In other words, $\phi_{G'}(\pi(R))$ is a single-element, and any product-one sequence $R$ has $1\in \pi(R)\subset G'$. Thus $\pi(S)\subset G'$ and $\pi(T)\subset G'$ follow from our hypotheses and,  consequently,    $$\phi_{G'}(\pi(T^{-1}\bdot S))=\phi_{G'}(\pi(T))^{-1}\phi_{G'}(\pi(S))= \{1\}^{-1} \{1\}=\{1\},$$ which  means $\pi(T^{[-1]}\bdot S)\subset G'$, as desired.
\end{proof}

The next lemma shows that a  product-one ordered sequence can have its terms cyclically shifted while preserving its product.

\begin{lemma}\label{permutaton-lemma}
Let $G$ be a group and let $S= g_1\bdot\ldots\bdot  g_{\ell}\in \Fc^*(G)$ be a product-one ordered sequence. Then  $ g_{j}\bdot\ldots \bdot g_{\ell}\bdot g_1\bdot\ldots\bdot g_{j-1}$ is also an product-one ordered sequence for every $j\in [1,\ell]$.
\end{lemma}

\begin{proof}
Let $S'= g_{\ell}\bdot g_1\bdot\ldots\bdot  g_{\ell-1}\in \Fc^*(G)$. Since $S$ has product-one, we have $$\pi(S')=g_{\ell}g_1\ldots g_{\ell-1}=g_\ell(g_1\ldots g_{\ell})g_{\ell}^{-1}=g_{\ell}\pi(S)g_{\ell}^{-1}=
g_{\ell} 1 g_{\ell}^{-1}=1.$$ Therefore $S'$ is also a product-one ordered sequence. Iterating this argument ${\ell}-j+1$ times shows that $ g_{j}\bdot\ldots \bdot g_{\ell}\bdot g_1\bdot\ldots\bdot g_{j-1}$ is a product-one ordered sequence, as desired.
\end{proof}

The next lemma
is proved by a standard argument. In particular, the statements for abelian groups are well-known. We provide the full proof so that the reader may become acquainted with the notation.

\begin{lemma}\label{dav-up-lemma}
Let $G$ be a group.
\begin{enumerate}
\item If $G$ is finite, then every ordered sequence $S \in \mathcal F^* (G)$ of length $|S| \ge |G|$ has a consecutive, product-one subsequence that is nontrivial. In particular, we have $\mathsf d (G) + 1 \le \mathsf D (G) \le |G|$.

\item $G$ is finite if and only if $\mathsf d (G)$ is finite.

\item If $G$ is finite abelian, then $\mathsf d (G)+1 = \mathsf D (G)$.

\item If $G$ is finite cyclic, then $\mathsf d (G)+1 = \mathsf D (G) = |G|$.
\end{enumerate}
\end{lemma}

\begin{proof}
1. Let $S  \in \mathcal F^* (G)$ be an ordered sequence of length $|S| = \ell \ge |G|$. For $j \in [1, \ell]$, we consider the elements $\pi \Big( S (1, j) \Big) \in G$. If $\pi \Big( S (1, j) \Big) = 1$ for some $j \in [1, \ell]$, then $S(1, j)$ is the desired consecutive, product-one subsequence. Otherwise, $\ell=|S| \ge |G|$ together with the pigeonhole principle guarantees that  there are $j,\, k \in [1,\ell]$ with $j < k$ and $\pi \Big( S (1, j) \Big) = \pi \Big( S (1, k) \Big)$, and then $S (j+1, k)$ is the desired consecutive, product-one subsequence.

It is now clear from Lemma \ref{lemma-no-consecutive-prod-1} that $\mathsf D(G)\leq |G|$.
If $S \in \mathcal F (G)$ is product-one free and $g \in \pi (S)$, then $S\bdot g^{-1} \in \mathcal A (G)$, and hence $\mathsf d (G)+1 \le \mathsf D (G)$.

2. By Item 1, it suffices to show that $\dd(G)=\infty$ when $|G|=\infty$.
Suppose that $G$ is infinite and let $S \in \mathcal F (G)$ be product-one free. If we can show that
there is a product-one free sequence of length $|S|+1$, then the assertion follows. Since $G$ is infinite but $|\pi(S)|\leq 2^{|S|}<\infty$, there is an element $g ^{-1}\in G \setminus \Pi(S)$, and we assert that $S \bdot g^{-1}$ is product-one free. Assume to the contrary that $1 \in \Pi (S \bdot g^{-1})$. Then, since $S$ is product-one free, there must exist a product-one subsequence $T\mid S\bdot g^{-1}$ with $g^{-1}\in \supp(T)$. Let $T^*=g_1 \bdot\ldots\bdot  g_{|T|}\in \Fc^*(G)$ be an ordering of $T$ such that $\pi(T^*)=1$. Then, by Lemma \ref{permutaton-lemma}, we can w.l.o.g. assume $T^*(1)=g^{-1}$, i.e., that $g_1=g^{-1}$, whence $g=g_1^{-1}=g_2\ldots g_{|T|}\in \pi(T\bdot (g^{-1})^{[-1]})\subset \Pi(S)$, contradicting that $g \in G \setminus \Pi(S)$. So $S \bdot g^{-1}$ is product-one free as claimed, completing the proof of Item 2.

3. Let $U\in \mathcal A(G)$ be an atom with $|U|=\mathsf D(G)$ and let $g\in \supp(G)$. Now consider $S=U\bdot g^{[-1]}$. Then $|S|=|U|-1=\mathsf D(G)-1$, and it suffices in view of Item 1 to show that $S$ is product-one free. Assuming this fails, then there must be a nontrivial, product-one subsequence $T\mid S$. Since $S\mid U$ is a proper subsequence, this ensures that $T\mid U$ is a proper, nontrivial, product-one subsequence of $U$. However, since $G$ is abelian with $T$ and $U$ both product-one sequences, we have $\pi(T^{[-1]}\bdot U)=\pi(T)\pi(T^{[-1]}\bdot U)=\pi(U)=1$, so that $U=T\bdot (T^{[-1]}\bdot U)$ is a factorization of $U$ into two nontrivial, product-one subsequences, contradicting that $U\in \mathcal A(G)$ is an atom.
Thus $S$ is product-one free, completing the proof of Item 3 as noted above.

4. If  $g \in G$ with $\ord (g) = |G|$, then the sequence $S = g^{[|G|-1]} \in \mathcal F (G)$ is product-one free,  hence  $|G|-1 \le \mathsf d (G)$, and thus the assertion follows from Item 1.
\end{proof}

We are not aware of a finite, non-abelian group with $\mathsf d (G)+1 = \mathsf D (G)$ (see also Lemma \ref{lemma-|G'|=2}).
Next we give a characterization for the large Davenport constant.

\begin{lemma}\label{Big-Dav-lemma}
Let $G$ be a finite group. Then $\mathsf D(G)$ is the smallest integer $\ell \in \mathbb N$ with the following property{\rm \,:}  for every sequence $S\in \Fc(G)$ of length $|S|\geq \ell$ and every $x\in \pi(S)$, there exists a nontrivial, product-one subsequence $T\mid S$ with $x\in \pi({T}^{[-1]}\bdot S)$ and $|T| \le \ell$.
\end{lemma}

\begin{proof}
First we show that $\mathsf D (G)$ has the required property.
Suppose $S\in \Fc(G)$ with $|S|\geq \mathsf D(G)$ and let $x\in \pi(S)$. Then $S\bdot x^{-1}\in \Fc(G)$ is a product-one sequence with length $|S\bdot x^{-1}|=|S|+1>\mathsf D(G)$. Repeatedly applying the definition of $\mathsf D(G)$ to the product-one sequence $S\bdot x^{-1}$ results in a factorization $S\bdot x^{-1}=T_1\bdot \ldots\bdot T_r$ with $T_i\in \mathcal A(G)$ atoms having \be\label{spillitmore}1\leq |T_i|\leq \mathsf D(G)\quad\mbox{ for $i\in [1,r]$}.\ee Since $|S\bdot x^{-1}|>\mathsf D(G)$, it follows that $r\geq 2$. Without restriction, we may assume $x^{-1}\in \supp(T_1)$, and then it is clear that $T_2\bdot\ldots \bdot T_r\mid S$ is a nontrivial, product-one subsequence (in view of $r\geq 2$) with \be\label{cowjumpting}(T_2\bdot\ldots\bdot T_r)^{[-1]}\bdot S=T_1\bdot( x^{-1})^{[-1]}.\ee Since $T_1$ is a product-one sequence, there is an ordering of the terms of $T_1$ having product $1$, say $T_1= x_1\bdot\ldots\bdot  x_n$ with $x_1\ldots x_n=1$. In view of Lemma \ref{permutaton-lemma}, we can cyclically shift the ordering so that $x^{-1}\in \supp(T_1)$ is the first term while preserving that the product of terms is $1$, i.e., we may w.l.o.g. assume $x_1=x^{-1}$. But now it is clear using \eqref{cowjumpting} that $$x=x_1^{-1}=x_2\ldots x_{n-1}\in \pi(T_1 \bdot x_1^{[-1]})=\pi(T_1\bdot (x^{-1})^{[-1]})=\pi((T_2\bdot\ldots\bdot T_2)^{[-1]}\bdot S)\subset \pi(T_2^{[-1]}\bdot S).$$ Thus, in view of \eqref{spillitmore}, it follows that  $T=T_2$ is the desired product-one subsequence of $S$.

To show that $\mathsf D(G)$ is the smallest integer with the desired property, consider an atom $U\in \mathcal A(G)$ with $|U|=\mathsf D(G)$ and an element $x^{-1}\in \supp(U)$, say $U = S \bdot  x^{-1}$ where $S \in \mathcal F (G)$ with $|S|=\mathsf D(G)-1$. Moreover, as argued above using Lemma \ref{permutaton-lemma}, we have $x\in \pi(S)$.
If by contradiction $S$ contained a nontrivial, product-one subsequence $T\mid S$ with $x\in \pi(T^{[-1]}\bdot S)$, then $U=T\bdot \big(T^{[-1]}\bdot S \bdot x^{-1}\big)=T\bdot \big(T^{[-1]}\bdot U\big)$ would be a factorization of $U$ into nontrivial, product-one subsequences, contradicting that $U\in \mathcal A(G)$ is an atom.
\end{proof}





Finally, we  need the concept of a \emph{setpartition}. Let $P$ be a set and let $Q$ be the set of \emph{finite and nonempty} subsets of $P$. The elements of $\mathscr S(P):=\mathcal F (Q)$ are called \emph{setpartitions over $P$}, and an $n$-{\it setpartition}, where $n\geq 0$,  is simply a setpartition $\mathscr A \in\mathscr S(P)$ having length $|\mathscr A| = n$. In other words, an $n$-setpartition $\mathscr A=A_1\bdot\ldots\bdot A_n \in \mathscr S(P)$ is a sequence of $n$ finite and \emph{nonempty} subsets $A_i\subset P$. The setpartition $\mathcal A\in \mathscr S(P)$ naturally partitions the sequence   $$\mathsf S(\mathscr A)=\bulletprod{i\in [1,n]}\bulletprod{a\in A_i} a\in \Fc(P) \,,$$
and $\mathscr A$ is said to have its terms being of \emph{as near equal a size as possible} if $$|A_i|\in \left\{\left\lfloor\frac{|\mathsf S(\mathscr A)|}{n}\right\rfloor,\,\left\lceil \frac{|\mathsf S(\mathscr A)|}{n}\right\rceil\right\}\quad\mbox{ for all $i\in [1,n]$} \,.$$
A sequence $S\in \Fc(P)$ is said to \emph{have an $n$-setpartition} if there is an $n$-setpartition $\mathscr A\in \mathscr S(P)$ with  $\mathsf S(\mathscr A)=S$.
The following is the standard existence result for setpartitions. It can be found in \cite[Proposition 10.2]{Gr13a}
or \cite{B-D-G-L03}.

\begin{lemma}\label{lemma-setpartition-exsitence}
Let $P$ be a set, let $S\in \Fc(P)$ be a sequence over $P$, and let $\ell\geq 0$ and $n\geq 1$ be integers. Then there is a subsequence $S'\mid S$ with $|S'|=\ell+n$ having an $n$-setpartition if and only if
\begin{align*}&|S|\geq \ell+n\quad\mbox{ and, \quad for every nonempty subset $X\subset P$ with $|X|\leq \frac{\ell-1}{n}+1$},\\ &\mbox{there are at most $|S|-\ell+(|X|-1)n$ terms of $S$ from $X$}.\end{align*} Moreover, if this is the case, then $S'$ has an $n$-setpartition with terms of as near equal a size as possible.

In particular, $S$ has an $n$-setpartition if and only if $\mathsf h(S)\leq n\leq |S|$, and if this is the case, then $S$ has an $n$-setpartition with terms of as near equal a size as possible.
\end{lemma}

\section{General Upper Bounds}\label{sec-upper-bounds}

We begin with the following upper bound of Olson and White \cite{Ol-Wh77} for the small Davenport constant.

\begin{theorem}\label{thm-olsonwhite-upper} Let $G$ be a finite, noncyclic group. Then $$\dd(G)\leq \frac12|G|$$ with equality if $G$ contains a cyclic, index $2$ subgroup.
\end{theorem}

The following gives an inductive upper bound for the large Davenport constant. We are indebted to an anonymous referee for having suggested the key idea at the heart of its proof.

\begin{theorem}\label{thm-induc-bound}
Let $G$ be a finite group and let $H\leq G$ be a subgroup. Then $$\mathsf D(G)\leq \mathsf D(H) |G:H|.$$
\end{theorem}

\begin{proof}The proof is similar to that of Lemma \ref{dav-up-lemma}.
We need to show that $|U|\leq \mathsf D(H) |G:H|$ for all $U\in \mathcal A(G)$.
Assume by contradiction that there is some $U\in \mathcal A(G)$ with $|U|>\mathsf D(H) |G:H|$. Since $U\in \mathcal A(G)$, there exists a product-one ordered sequence $U^*\in \Fc^*(G)$ with $[U^*]=U$.


For every $j \in [1, |U|]$, we consider the elements $\pi \Big( U^* (1, j) \Big) \in G$.
Since $|U|>\mathsf D(H) |G:H|$, the pigeonhole principle guarantees that there exists some left $H$-coset, say $gH$, for which $\pi \Big( U^* (1, j) \Big) \in gH$ holds for at least $\mathsf D(H)+1$ values of $j \in [1, |U|]$. Let $j_1<j_2<\ldots <j_r$, where $r\geq \mathsf D(H)+1$, be  all those indices $j_i \in [1, |U|]$ with $\pi \Big( U^* (1, j_i) \Big) \in gH$. Our next goal is to show that, by cyclically shifting the ordered sequence $U^*$, we can w.l.o.g. assume $j_r=|U|$.

Consider the ordered sequence ${U'}^*=U^*(j_r+1,|U|)\bdot U^*(1,j_r)\in \Fc^*(G)$. Clearly, we have $[{U'}^*]=[U^*]=U$. However, we also have
\begin{align*}\pi\Big({U'}^*\Big)&=
\pi\Big(U^*(j_r+1,|U|)\Big)\pi\Big(U^*(1,j_r)\Big)=
\pi\Big(U^*(j_r+1,|U|)\Big)\pi\Big(U^*\Big)
\pi\Big(U^*(j_r+1,|U|)\Big)^{-1}\\
&=\pi\Big(U^*(j_r+1,|U|)\Big)1 \pi\Big(U^*(j_r+1,|U|)\Big)^{-1}=1.
\end{align*} Thus ${U'}^*$ is a product-one ordered sequence with $[{U'}^*]=U$. Moreover, letting $s=|U^*(j_r+1,|U^*|)|=|U^*|-j_r$ and $g'=\pi\Big(U^*(j_r+1,|U^*|)\Big)$, we see (in view of the definition of the $j_i$) that
$$\pi \Big({U'}^*(1,j_i+s)\Big)=g'\pi\Big(U(1,j_i)\Big)\in
g'gH\quad\mbox{ for all $i\in [1,r]$}.$$ Consequently, repeating the above arguments using the ordered sequence ${U'}^*$ in place of $U^*$ allows us to w.l.o.g. assume $j_r=|U^*|$. But then $1=\pi(U^*)=\pi\Big(U^*(1,|U^*|)\Big)=\pi\Big(U^*(1,j_r)\Big)\in gH$ forces  $gH=H$.
Thus we now have  \be\label{seenat}\pi\Big( U^* (1, j_i) \Big) \in H\quad\mbox{ for $i\in [1,r]$}.\ee

Let $U^*_i=U^*(j_{i-1}+1,j_i)\in\Fc^*(G)$ for $i\in [1,r]$, where $j_0:=0$. Since $j_r=|U|$, we have \be\label{stillhop}U^*_1\bdot\ldots\bdot U^*_r=U^*.\ee In view of \eqref{seenat}, we have $$\pi(U^*_1), \,\pi(U^*_1\bdot U^*_2),\,\pi(U^*_1\bdot U^*_2\bdot U^*_3),\ldots,\pi(U^*_1\bdot\ldots\bdot U^*_r)\in H.$$ A simple inductive argument now shows \be\label{seenet}\pi(U^*_i)\in H\quad\mbox{ for all $i\in [1,r]$}.\ee

In view of \eqref{stillhop} and \eqref{seenet}, consider the sequence $S= \pi(U^*_1)\bdot\ldots \bdot  \pi(U^*_r)\in \Fc(H)$. Since $\pi(U^*_1)\ldots \pi(U^*_r)=\pi(U^*_1\bdot\ldots\bdot U^*_r)=\pi(U^*)=1$, we see that $S\in \mathcal B(H)$. However, since $|S|=r\geq \mathsf D(H)+1$, the definition of $\mathsf D(H)$ ensures that we have some factorization of $S$, say $$S=\bulletprod{i\in I} \pi(U^*_i)\bulletprod{i\in [1,r]\setminus I} \pi(U^*_i),$$
where $I\subset [1,r]$, with both $\bulletprod{i\in I} \pi(U^*_i)$ and $\bulletprod{i\in [1,r]\setminus I} \pi(U^*_i)$ nontrivial, product-one sequences over $H\leq G$. But then it is clear that both $\left[\bulletprod{i\in I} U^*_i\right]$ and $\left[\bulletprod{i\in [1,r]\setminus I}U^*_i\right]$ are nontrivial, product-one sequences over $G$, whence the factorization (in view of \eqref{stillhop}) $$U=[U^*]=\left[ \bulletprod{i\in I} U^*_i\right]\bdot \left[\bulletprod{i\in [1,r]\setminus I}U^*_i\right]$$ contradicts that $U\in \mathcal A(G)$ is an atom, completing the proof.
\end{proof}

A similar argument to that of Theorem \ref{thm-induc-bound} gives the following result.

\begin{theorem}\label{thm-induc-bound-extra}
Let $G$ be a finite group and let $H\lhd G$ be a normal subgroup with $H\cap G'=\{1\}$, where $G'=[G,G]\leq G$ is the commutator subgroup of $G$. Then $$\mathsf D(G)\leq \mathsf D(H) \mathsf D(G/H).$$
\end{theorem}

\begin{proof}
Assume by contradiction that there is some atom $U\in \mathcal A(G)$ with $|U|>\mathsf D(H) \mathsf D(G/H)$.
Since $U$ is a product-one sequence, we have $1\in \pi(U)\subset G'$. Since $|U|>\mathsf D(H) \mathsf D(G/H)$, repeatedly applying   Lemma \ref{Big-Dav-lemma} to the product-one sequence $\phi_H(U)\in\Fc(G/H)$ taking $x=1$ each time yields a factorization $$U=U_1\bdot\ldots\bdot U_r\quad\mbox{ with }\quad \pi(U_i)\cap H\neq \emptyset\quad\mbox{ for $i\in [1,r]$ } \quad\und \quad r> \mathsf D(H).$$
Since $\pi(U_i)\cap H\neq \emptyset$ for $i\in [1,r]$, it follows that each $U_i\in \Fc(G)$ has an ordering $U^*_i\in \Fc^*(G)$, so $[U^*_i]=U_i$, such that $\pi(U^*_i)\in H$.
As a result, $\pi(U^*_1)\ldots \pi(U^*_r)\in H$. However, we also have $$\pi(U^*_1)\ldots \pi(U^*_r)\in \pi([U_1^*\bdot\ldots\bdot U^*_r])=\pi(U)\subset G'.$$ Thus, in view of the hypothesis $H\cap G'=\{1\}$, it follows that
$\pi(U^*_1)\ldots \pi(U^*_r)=1$. But this shows that  $$U':= \pi(U^*_1)\bdot\ldots\bdot \pi(U^*_r)\in \Fc(H)$$ is a product-one sequence of length $r>\mathsf D(H)$. Consequently, the definition of $\mathsf D(H)$ ensures that there is a factorization
\[
U'=\left(\bulletprod{i\in I} \pi(U^*_i)\right)\bdot \left(
\bulletprod{i\in [1,r]\setminus I} \pi(U^*_i)\right)
\]
 with $\left(\bulletprod{i\in I} \pi(U^*_i)\right)$ and $\left(
\bulletprod{i\in [1,r]\setminus I} \pi(U^*_i)\right)$ both nontrivial, product-one subsequences of $U'$, where $I\subset [1,r]$. But then $U=\left[\bulletprod{i\in I}U^*_i\right]\bdot \left[
\bulletprod{i\in [1,r]\setminus I}U^*_i\right]$ is a factorization of $U$ into $2$ nontrivial, product-one subsequences, contradicting that $U\in \mathcal A(G)$ is an atom.
\end{proof}

Next, we give an upper bound in the case when $G$ is nearly abelian.

\begin{lemma}\label{lemma-|G'|=2}
Let $G$ be a finite group with commutator subgroup $G'=[G,G]\leq G$. Suppose $|G'|\leq 2$. Then
$$\mathsf D(G)\leq \mathsf d(G)+|G'|.$$
\end{lemma}

\begin{proof} If $|G'|=1$, then $G$ is abelian and $\mathsf d(G)+1=\mathsf D(G)$ holds by Lemma \ref{dav-up-lemma}.  Therefore we may assume $|G'|=2$.
Assume by contradiction that we have an atom  $U\in \mathcal A(G)$  with \be\label{longer}|U|=\mathsf D(G)\geq \mathsf d(G)+|G'|+1=\mathsf d(G)+3.\ee

If all the terms of $U$ commute with each other, then $\supp(U)$ generates an abelian group, whence $$|U|\leq \mathsf D(\la \supp(U)\ra)=\dd(\la \supp(U)\ra)+1\leq \dd(G)+1,$$ contrary to \eqref{longer}. Therefore we may assume there are terms $x,\,y\in \supp(U)$ which do not commute with each other: $xy\neq yx$. Let $T= x\bdot y\in \Fc(G)$ be the subsequence consisting of these $2$ terms. Since the terms of $T$ do not commute with each other, we have  $|\pi(T)|=2=|G'|$, and since $\pi(T)$ must be contained in a $G'$-coset (as noted in Section \ref{sec-notation}), this ensures that $\pi(T)$ is an entire $G'$-coset. In view of  \eqref{longer}, we have  $|T^{[-1]}\bdot U|=|U|-2\geq \mathsf d(G)+1$. Thus the definition of $\mathsf d(G)$ ensures that there is a nontrivial, product-one subsequence $R\mid T^{[-1]}\bdot U$. From Lemma \ref{lemma-G'}, we know that $\pi(R^{[-1]}\bdot U)\subset G'$. Thus, since $|\pi(T)|=|G'|$ with $T\mid R^{[-1]}\bdot U$, we conclude that $\pi(R^{[-1]}\bdot U)=G'$. In particular, $1\in\pi(R^{[-1]}\bdot U)$, meaning $R^{[-1]}\bdot U$ is also a product-one subsequence, which is nontrivial in view of $T\mid R^{[-1]}\bdot U$ and $|T|=2$. But now $U=R\bdot (R^{[-1]}\bdot U)$ is a factorization of $U$ into $2$ nontrivial, product-one subsequences, contradicting that $U\in \mathcal A(G)$ is an atom.
\end{proof}


\section{Some Tools from Additive Theory}\label{sec-add}

In this section, we present the results from Additive Theory needed for Theorem \ref{thm-index2}.
To simplify notation, all groups in this section will be abelian and written additively.
We begin with the classical Cauchy-Davenport Theorem \cite[Theorem 6.2]{Gr13a} .

\begin{theorem}[Cauchy-Davenport Theorem]
Let $G$ be an abelian group of prime order $p$ and let $A_1,\ldots,A_n\subset G$ be nonempty subsets. Then
$$|\Sum{i=1}{n}A_i|\geq \min\{p,\; \Sum{i=1}{n}|A_i|-n+1\}.$$
\end{theorem}

Next, we state the following special case of either the DeVos-Goddyn-Mohar Theorem  or the Partition Theorem (see \cite[Chapters 13 and 14]{Gr13a} or \cite{De-Go-Mo09a}).

\begin{theorem}\label{DGM-devos-etal-thm}
Let $G$ be an abelian group, let $S\in \Fc(G)$ be a sequence, let  $n\in [1,|S|]$, and let  $H=\HH(\Sigma_n(S))$. Then \be\label{DGM-bound}|\Sigma_n(S)|\geq \Big(\Summ{g\in G/H}\min\{n,\,\vp_{g}\big(\phi_H(S)\big)\}-n+1\Big)|H|.\ee
\end{theorem}

For the proof of Theorem \ref{thm-index2}, the case when $G$ is isomorphic to the dicyclic group $Q_{4p}$ of order $4p$ with $p\geq 3$ prime proves to be particularly difficult.
One of the key ideas for handling this case is to reduce the basic product-one question for  the non-abelian group $Q_{4p}$ into a more complicated zero-sum question over the abelian group $C_{2p}$: Lemma \ref{lem-factorizationI}.
However, we first need some additional notation.

Given an additively written, abelian group $G$, we let $$2G=\{2g:\; g\in G\}\leq G$$ denote the homomorphic image of $G$ under the multiplication by $2$ homomorphism. Likewise, given a sequence $S=g_1\bdot \ldots\bdot  g_\ell\in \Fc(G)$, we let $$2S=2g_1\bdot \ldots\bdot 2g_\ell\in \Fc(2G).$$

For the following lemma, we will make use of the fact that \be \label{key-equivalence} \bigcup_{T\mid S,\; |T|=n}\left(\sigma(T)-\sigma(S\bdot T^{[-1]})\right)=\Sigma_n(2S)-\sigma(S)\ee for any sequence $S\in \Fc(G)$ with $|S|\geq n\geq 1$---the equality follows routinely from the definitions involved.

We remark that Lemma \ref{lem-factorizationI} remains true without assuming $p\geq 2$ is prime. However, the proof is much more technical and requires a somewhat detailed case distinction for defining and dealing with the subsequence $S'$ in the proof. As we only need the case when $p$ is prime, we have opted to present the simplified proof. Moreover, we will actually show  Lemma
\ref{lem-factorizationI} holds with $|U_1|=|U_2|\leq 2$.

\begin{lemma}\label{lem-factorizationI}
Let $G$ be a cyclic group of order $|G| = 2p$ with $p\geq 2$ prime, let $x\in G$ be the unique element with $\ord(x)=2$,  and let $S\in \Fc(G)$ be a sequence of even length $|S|\geq  2p+4$. Suppose there is a factorization $$S=T_1\bdot T_2\quad\mbox{ with }\quad  |T_1|=|T_2|=\frac{1}{2}|S|\quad\und\quad \sigma(T_1)-\sigma(T_2)=|T_1|x,$$ where $T_1,\,T_2\in \Fc(G)$. Then
there is a factorization $S=U_1\bdot U_2\bdot V_1\bdot V_2$, where $U_1,\,U_2,\,V_1,\,V_2\in \Fc(G)$ are nontrivial, such that
\be\label{good-fact}
|U_1|=|U_2|,\quad |V_1|=|V_2|,\quad \sigma(U_1)-\sigma(U_2)=|U_1|x \quad \und \quad \sigma(V_1)-\sigma(V_2)=|V_1|x \,.
\ee
\end{lemma}

\begin{proof} Let $|S|=2\ell\geq 2p+4\geq 8$, so that \be\label{Tlong}|T_1|=|T_2|=\ell\geq p+2\geq 4.\ee Note $x=-x$ and $$|T_1|x=\left\{
                                                                              \begin{array}{ll}
                                                                                0, & \hbox{if $|T_1|=\ell$ is even} \\
                                                                                x, & \hbox{if $|T_1|=\ell$ is odd.}
                                                                              \end{array}
                                                                            \right.$$

If $g\in \supp(T_1)$ with $g+x\in \supp(T_2)$ for some $g\in G$, then the
lemma follows setting $U_1=g$, \ $U_2=x+g$, \ $V_1=T_1\bdot g^{[-1]}$ and $V_2=T_2\bdot (x+g)^{[-1]}$---in view of the hypotheses $|T_1|=|T_2|=\ell\geq 2$ and $\sigma(T_1)-\sigma(T_2)=|T_1|x$. Likewise, if there is some $g\in G$ with $\vp_g(T_1),\,\vp_g(T_2)\geq 2$, then the lemma follows setting $U_1=U_2=g^{[2]}$, \ $V_1=T_1\bdot g^{[-2]}$ and $V_2=T_2\bdot g^{[-2]}$---in view of $|T_1|=|T_2|=\ell\geq 3$. Therefore, we may assume
\ber\label{extra-ded} (\supp(T_1)+x)\cap \supp(T_2)=\emptyset \quad\und
\\
\min\{\vp_g(T_1),\,\vp_g(T_2)\}\leq 1\quad \mbox{ for all $g\in G$.}\label{disjoint}
\eer
In particular, \be\label{Benhie}\mathsf h(S)=\mathsf h(T_1\bdot T_2)\leq \max\{|T_1|+1,\,|T_2|+1\}=\ell+1.\ee
Since $G\cong C_{2p}$, given any $\alpha\in G$,
there are exactly $2$ distinct elements $g,\,h\in G$ such that $2g=2h=\alpha$.

Observing that it suffices to prove the lemma for any translated sequence $-g+S$, where $g\in G$ (the conclusions and hypotheses of the lemma are translation invariant), we may w.l.o.g.
translate our sequence $S$ so that \be\label{0-mod2-mult}\vp_{0}(2S)=\mathsf h(2S).\ee
Note that $$2S\in \Fc(2G)\quad \mbox{ with } \quad 2G\cong C_p.$$

If $\mathsf h(2S)\leq 2$, then \eqref{Tlong} gives $2p+4\leq 2\ell=|S|=|2S|\leq \mathsf h(2S)|2G|\leq 2p$, a contradiction. Therefore we have \be\label{dogl}\vp_{0}(2S)=\mathsf h(2S)\geq 3.\ee
By translating by $-x$ if need be, which  preserves \eqref{0-mod2-mult} since $2x=0$, we may w.l.o.g. assume \be\label{0more}\vp_0(S)\geq \vp_x(S).\ee
We distinguish two cases.

\subsection*{Case 1:} $x\in \supp(S)$.

In view of $x\in \supp(S)$ and \eqref{0more}, we have $0,\,x\in \supp(S)$. Set $S'=S\bdot 0^{[-1]}\bdot x^{[-1]}$ and $\ell'=\frac12|S'|=\ell-1$. Since $0,\,x\in\supp(S)$, it follows from   \eqref{extra-ded} that either $\supp(T_1)\cap \{0,x\}=\emptyset$ or $\supp(T_2)\cap \{0,x\}=\emptyset$. Combining this with \eqref{0-mod2-mult}, we conclude that \be\label{goodcase} \mathsf h(2S)=\vp_0(2S)\leq \max\{|T_1|,\,|T_2|\}=\ell.\ee


We will show that \be\label{finalthing} \ell'x\in \bigcup_{T\mid S',\; |T|=\ell'}\left(\sigma(T)-\sigma(S'\bdot T^{[-1]})\right)=\Sigma_{\ell'}(2S')-\sigma(S'),\ee where the equality above was noted in \eqref{key-equivalence}. Once \eqref{finalthing} is established, we will know there exists some subsequence $T\mid S'=S\bdot 0^{[-1]}\bdot x^{[-1]}$ such that $$\ell'=|T|=2\ell'-\ell'=|S'|-|T|=|S'\bdot T^{[-1]}|\quad \und\quad \sigma(T)-\sigma(S'\bdot T^{[-1]})=\ell'x=|T|x,$$ whence the lemma will follow setting $U_1=0$, \ $U_2= x$, \ $V_1=T$ and $V_2=S'\bdot T^{[-1]}$. Thus it remains to establish \eqref{finalthing} for the sequence $S'$ to complete Case 1. For this, we apply Theorem \ref{DGM-devos-etal-thm} to $\Sigma_{\ell'}(2S')$.


In view of the hypotheses $S=T_1\bdot T_2$ with $\sigma(T_1)-\sigma(T_2)=|T_1|x=\ell x$, we know $$\sigma(S)=2\sigma(T_2)+\ell x=2\sigma(T_2)+\ell'x+x.$$ Thus  $$\sigma(S')+\ell'x=\sigma(S)-x+\ell' x=2\sigma(T_2)+2\ell'x=2\sigma(T_2)\in 2G.$$ Consequently, if $\Sigma_{\ell'}(2S')=2G$, then  $2\sigma(T_2)=\sigma(S')+\ell'x\in \Sigma_{\ell'}(2S')$ follows, yielding \eqref{finalthing}, as desired. Therefore we may assume \be\label{stavet}|\Sigma_{\ell'}(2S')|\leq |2G|-1=p-1.\ee
Consequently, since $2G\cong C_p$ has no nontrivial, proper subgroups,  we must have $\mathsf H(\Sigma_{\ell'}(2S'))$ trivial.
Since $\mathsf H(\Sigma_{\ell'}(2S'))$ is trivial and $\ell'=\ell-1\geq p+1$ (by \eqref{Tlong}), Theorem \ref{DGM-devos-etal-thm} will contradict \eqref{stavet} if
$2S'$ contains $2$ distinct terms each having multiplicity at least $\ell'+1$. Thus there can be at most one distinct term with multiplicity at least $\ell'+1$ in $2S'$. Furthermore,  Theorem \ref{DGM-devos-etal-thm} will again contradict \eqref{stavet} unless such  a term from $2S'$ exists having  multiplicity at least $2\ell'-p+2=2\ell-p\geq \ell+2$, where the inequality follows from \eqref{Tlong}. However the latter contradicts \eqref{goodcase} in view of the trivial inequality $\mathsf h(2S')\leq \mathsf h(2S)$, completing Case 1.

\subsection*{Case 2:} $x\notin \supp(S)$.

Since $x\notin \supp(S)$, it follows from \eqref{dogl} that $$\vp_0(S)=
\vp_0(S)+\vp_x(S)=\mathsf v_0(2S)\geq 3.$$
If $\mathsf h(S\bdot 0^{[-2]})\leq 1$, then it follows in view of the case hypothesis that $$2p+2\leq |S|-2=|S\bdot 0^{[-2]}|\leq |G\setminus\{x\}|=2p-1,$$ a contradiction. Therefore we must instead have some $g\in \supp(S\bdot 0^{[-2]})$ with $\vp_g(S\bdot 0^{[-2]})\geq 2$, allowing us to define $S':=S\bdot 0^{[-2]}\bdot g^{[-2]}$. Let $\ell'=\ell-2=\frac12|S'|$.
Note that $S'$ is nontrivial in view of $|S|=2\ell\geq 2p+4\geq 8$. For the moment, $g\in \supp(S\bdot 0^{[-2]})$ is an arbitrary element with $\vp_g(S\bdot 0^{[-2]})\geq 2$. We will choose $g$ more carefully later in the proof.

Next, we will show that \be\label{finalthingII} \ell'x\in \bigcup_{T\mid S',\; |T|=\ell'}\left(\sigma(T)-\sigma(S'\bdot T^{[-1]})\right)=\Sigma_{\ell'}(2S')-\sigma(S'),\ee where the equality above was noted in \eqref{key-equivalence}. Once \eqref{finalthingII} is established, we will know there exists some subsequence $T\mid S'=S\bdot 0^{[-2]}\bdot g^{[-2]}$ such that $$\ell'=|T|=2\ell'-\ell'=|S'|-|T|=|S'\bdot T^{[-1]}|\quad \und\quad \sigma(T)-\sigma(S'\bdot T^{[-1]})=\ell'x=|T|x,$$ whence the lemma will follow setting $U_1=0\bdot g$, \ $U_2= 0\bdot g$, \ $V_1=T$ and $V_2=S'\bdot T^{[-1]}$. Thus it remains to establish \eqref{finalthingII} for the sequence $S'$ to complete Case 2. For this, we apply Theorem \ref{DGM-devos-etal-thm} to $\Sigma_{\ell'}(2S')$.


\bigskip

In view of the hypotheses $S=T_1\bdot T_2$ with $\sigma(T_1)-\sigma(T_2)=|T_1|x=\ell x$, we know $$\sigma(S)=2\sigma(T_2)+\ell x=2\sigma(T_2)+(\ell-2)x=2\sigma(T_2)+\ell'x.$$ Thus  $$\sigma(S')+\ell'x=\sigma(S)-2g+\ell' x=2\sigma(T_2)-2g+2\ell'x=2\sigma(T_2)-2g\in 2G.$$ Consequently, if $\Sigma_{\ell'}(2S')=2G$, then  $\sigma(S')+\ell'x\in \Sigma_{\ell'}(2S')$ follows, yielding \eqref{finalthingII}, as desired. Therefore we may assume \be\label{stavetII}|\Sigma_{\ell'}(2S')|\leq |2G|-1=p-1.\ee Consequently, since $2G\cong C_p$ has no nontrivial, proper subgroups,  we must have $\mathsf H(\Sigma_{\ell'}(2S'))$ trivial, in which case Theorem \ref{DGM-devos-etal-thm} yields
 \be\label{the-thing}|\Sigma_{\ell'}(2S')|\geq \Summ{y\in 2G}\min\{\ell',\,\vp_y(2S')\}-\ell'+1.\ee

Since $\ell'=\ell-2\geq p$ holds by \eqref{Tlong}, we see that if there are $2$ distinct terms of $2S'$ each having multiplicity at least $\ell'+1$, then \eqref{the-thing} will contradict \eqref{stavetII}. Therefore, there is at most one distinct term of $2S'$ having multiplicity at least $\ell'+1$. Moreover, \eqref{the-thing} will again contradict \eqref{stavetII} unless such a term of $2S'$ exists having multiplicity at least $2\ell'-p+2=2\ell-p-2$. Thus \be\label{highmult}\mathsf h(2S')\geq 2\ell-p-2\geq \ell,\ee where the latter inequality follows from \eqref{Tlong}.
 In view of our case hypothesis, \eqref{0-mod2-mult} and \eqref{Benhie}, it follows that $$\mathsf h(2S')\leq \mathsf h(2S)=\vp_0(2S)=\vp_0(S)\leq \ell+1.$$

Suppose $\mathsf h(2S)=\vp_0(2S)=\ell+1$. Then all nonzero elements will have multiplicity at most $|2S|-\ell-1=\ell-1$ in $2S$, and thus also in $2S'$, while $\vp_0(2S')\leq \vp_0(2S)-2=\ell-1$ follows in view of $S'=S\bdot 0^{[-2]}\bdot g^{[-2]}$. In such case, it follows that  $\mathsf h(2S')\leq \ell-1$, contradicting \eqref{highmult}. So we must have $\mathsf h(2S)=\vp_0(2S)\leq \ell$.
On the other hand, if $\mathsf h(2S)\leq \ell-1$, then \eqref{highmult} will again be contradicted in view of the trivial inequality $\mathsf h(2S')\leq \mathsf h(2S)$.
So we conclude that $$\mathsf h(2S)=\vp_0(2S)=\ell.$$

Now $\vp_0(2S')\leq \vp_0(2S)-2= \ell-2$. Thus \eqref{highmult} ensures that there must be a nonzero element having multiplicity at least $\ell$ in $2S'$, and thus also in $2S$. Since $0$ also has multiplicity at least $\ell$ in $2S$ with $|2S|=|S|=2\ell$, this is only possible if $|\supp(2S)|=2$ with both elements from $\supp(2S)$ having multiplicity $\ell$ in $2S$. As a result, since $\ell\geq p+2\geq 3$, the pigeonhole principle guarantees that we can take $g$ with $2g\neq 0$ when defining $S'=S\bdot 0^{[-2]}\bdot g^{[-2]}$, whence $\vp_0(2S')= \vp_0(2S)-2= \ell-2$ and $\vp_{2g}(2S')= \vp_{2g}(2S)-2= \ell-2$ follow, contradicting \eqref{highmult} for the final time.
\end{proof}


\section{Groups with a Cyclic, Index $2$ Subgroup}\label{sec-index2}

In this section, we determine the large Davenport constant of all finite groups containing a cyclic, index $2$ subgroup.
Despite the simple formulation of Theorem \ref{thm-index2}, we will need some specialized information regarding the isomorphism classes of such groups. Thus we  summarize their classification in a form suitable for our needs. The main result is Theorem   \ref{thm-full-index2-class}, which is taken from a recent monograph by Jones, Kwak, and   Xu \cite[Section 3.4.3]{Jo-Kw-Xu13a}. We start with a lemma which is slightly stronger than the classical result by H\"older. The lemma follows from  the characterization given in the above monograph; we have pulled it out for clarity. H\"older's Theorem can be found in \cite[Chap. III, \S 7]{Za56a} or \cite[Chapter 7]{Jo90a}.

\begin{lemma}\label{Generic-Pres-index2}
Let $G$ be a finite group of order $|G|=2n=2^{s+1}m$, where $\gcd(2,m)=1$, \ $s\geq 0$, \ $m\geq 1$, and $n=2^sm$. Suppose $G$ has a cyclic, index $2$ subgroup. Then $G$ has a presentation of one of the following  forms:
\begin{itemize}
\item[(A)] $G=\la \alpha,\,\tau\mid \alpha^{n}=1,\quad\alpha\tau=\tau\alpha^r,\quad\tau^2=1\ra$,\\
\item[(B)] $G= \la \alpha, \, \tau\mid \alpha^{n}=1,\quad \alpha\tau=\tau\alpha^r,\quad \tau^2=\alpha^{\frac12 n}\ra,$\quad or \\

\item[(C)] $G=\la \alpha,\,\tau\mid \alpha^{n}=1,\quad\alpha\tau=\tau\alpha^r,\quad\tau^2=\alpha^m\ra$
\end{itemize}
for some $r\in [1,n]$ with (B) only possible if $s\geq 1$. In particular,
\[
G=\{1,\alpha,\alpha^2,\ldots,\alpha^{n-1}\}\cup \{\tau,\tau\alpha,\tau\alpha^2,\ldots,\tau
\alpha^{n-1}\} \,.
\]
\end{lemma}

Of course, not all values of $r\in [1,n]$ are possible nor necessarily give rise to non-isomorphic groups. However, throughout this section, we will use the format given by Lemma \ref{Generic-Pres-index2} for $G$, saying that $G$ has type (A)  if it has a presentation given by (A) in Lemma \ref{Generic-Pres-index2},  and likewise defining types (B) and (C). Note that if $G$ is of  type (C) with $r=1$, then $\ord(\tau\alpha)=2n$, which corresponds to when $G$ is cyclic. Also, when $s=0$, type (C) coincides with type (A), and when $s=1$, type (C) coincides with type (B). Type (C) is really only needed when $s\geq 2$, but it will be convenient to state Lemma \ref{Generic-Pres-index2} with this slight amount of overlap between types.

\bigskip

In order to unify the notation in the proofs and statements of theorems in this section, we list a set of assumptions regarding hypotheses and notation that we will use throughout this section.
The importance of the parameters $n^-$, \ $n^+$, \ $m^-$ and $m^+$ will become apparent later in the section.

\bigskip

\begin{itemize}
\item[]\begin{center} {\bf  General Assumptions for Section \ref{sec-index2}} \end{center}

\medskip
\item $G$ is a finite group of order $|G|=2n=2^{s+1}m$, where $\gcd(2,m)=1$, \ $s\geq 0$, \ $m\geq 1$, \ and  $n=2^sm$.
\item $G$ has a cyclic, index $2$ subgroup, notated as in Lemma \ref{Generic-Pres-index2}, with parameter $r\in [1,n]$.
\item $G'=[G,G]\leq G$ is the commutator subgroup of $G$.
\item $P\leq G$ is a Sylow $2$-subgroup of $G$.
\item $n^-=\gcd(r-1,n)$  \ and \ $n=n^+n^-$.
\item $m^-=\gcd(r-1,m)$ \ and  \ $m^+=\gcd(r+1,m)$.
\end{itemize}

\bigskip

We continue with the characterization for $2$-groups, which can be found in many standard texts (e.g., \cite[Theorem 1.2]{Be08a}). The general case (Theorem \ref{thm-full-index2-class}) follows by routine arguments from the $2$-group case.

\begin{lemma}\label{lem-2-group-class}
Let $G$ satisfy the General Assumptions for Section \ref{sec-index2}. Suppose $G$ is a $2$-group, so $m=1$ and $G=P$. Then $G$ is isomorphic to one of the following non-isomorphic  groups.

\begin{itemize}
\item[(i)] $s\geq 0$ and $G$ is a cyclic group: $$G\cong C_{2^{s+1}}= \la\alpha,\,\tau\mid \alpha^{2^s}=1,\quad\alpha \tau=\tau \alpha,\quad \tau^2=\alpha\ra.$$
\item[(ii)] $s\geq 1$ and $G$ is an abelian but non-cyclic group: $$G\cong C_2\times C_{2^s}=\la\alpha,\,\tau\mid \alpha^{2^s}=1,\quad\alpha \tau=\tau \alpha,\quad \tau^2=1\ra.$$
\item[(iii)] $s\geq 2$ and $G$ is a dihedral group: $$G\cong D_{2^{s+1}}=\la\alpha,\,\tau\mid \alpha^{2^s}=1,\quad\alpha \tau=\tau \alpha^{-1},\quad \tau^2=1\ra.$$

\item[(iv)] $s\geq 2$ and $G$ is a generalized quaternion group: $$G\cong Q_{2^{s+1}}=\la\alpha,\,\tau\mid \alpha^{2^s}=1,\quad\alpha \tau=\tau \alpha^{-1},\quad \tau^2=\alpha^{2^{s-1}}\ra.$$

\item[(v)] $s\geq 3$ and $G$ is a semi-dihedral group: $$G\cong SD_{2^{s+1}}=\la\alpha,\,\tau\mid \alpha^{2^s}=1,\quad \alpha \tau=\tau \alpha^{-1+2^{s-1}},\quad \tau^2=1\ra.$$
\item[(vi)] $s\geq 3$ and $G$ is an ordinary meta-cyclic group: $$G\cong M_{2^{s+1}}=\la\alpha,\,\tau\mid \alpha^{2^s}=1,\quad \alpha \tau=\tau \alpha^{1+2^{s-1}},\quad \tau^2=1\ra.$$
\end{itemize}
\end{lemma}

In view of Lemma \ref{lem-2-group-class}, given a finite $2$-group $P$ of order $2^{s+1}$ having a cyclic, index $2$ subgroup, we let $\rho(P)\in [1,2^s]$ be the value of $r$ in its presentation given by Lemma \ref{Generic-Pres-index2}, i.e.,
$$\rho(P)=\left\{
            \begin{array}{ll}
              1, & \hbox{for $P$ given by Lemma \ref{lem-2-group-class}(i)(ii) with $s\geq 0$} \\
              -1+2^s, & \hbox{for $P$ given by Lemma \ref{lem-2-group-class}(iii)(iv) with $s\geq 2$} \\
              -1+2^{s-1}, & \hbox{for $P$ given by Lemma \ref{lem-2-group-class}(v) with $s\geq 3$}\\
1+2^{s-1}, & \hbox{for $P$ given by Lemma \ref{lem-2-group-class}(vi) with $s\geq 3$.}
            \end{array}
          \right.$$

The full classification of finite groups having a cyclic, index $2$ subgroup is then the following.

\begin{theorem}\label{thm-full-index2-class} Let $G$ satisfy the General Assumptions for Section \ref{sec-index2}. Then the Sylow $2$-group $P$ is of one of the six types (i)--(vi) given by Lemma \ref{lem-2-group-class} and $$r\in [1,n]\quad\mbox{ satisfies }\quad r^2\equiv 1\mod m\quad\und \quad r\equiv \rho(P)\mod 2^s.$$ Furthermore,
\begin{itemize}
\item[1.] If $P$ is of type (ii), (iii), (v) or (vi), then $G$ has type (A) in Lemma \ref{Generic-Pres-index2}.
\item[2.] If $P$ is of type (iv), then $G$ has type (B) in Lemma \ref{Generic-Pres-index2}.
\item[3.] If $P$ is of type (i), then $G$ has type (C) in Lemma \ref{Generic-Pres-index2}.
 \end{itemize}
Different allowed values of $r\in [1,n]$ correspond to non-isomorphic groups, and any group described above indeed has a cyclic, index $2$ subgroup.
\end{theorem}

From Theorem \ref{thm-full-index2-class}, we see that the parameter $r\in [1,n]$ must satisfy the equation \be\label{r-defining-n}(r+1)(r-1)=r^2-1\equiv 0\mod m.\ee  Now consider a prime  $p$ dividing $m$. Since $p$ must be odd (as $m$ is odd), either $\gcd(r+1,p)=1$ or $\gcd(r-1,p)=1$. Thus \eqref{r-defining-n} implies that  either $r+1\equiv 0\mod p^{\vp_p(m)}$ or $r-1\equiv 0\mod p^{\vp_p(m)}$. This means that we can factor \begin{align}\label{n-plus-minus-def}&m=m^+m^-\quad\mbox{ with }\quad
\gcd(m^+,m^-)=1,\quad\mbox{ where}\\ \nn&m^+\geq 1\quad \mbox{ contains all those primes $p\mid m$ with $r+1\equiv 0\mod p$}\quad  \und \\\nn&m^-\geq 1 \quad\mbox{ contains all those primes $p\mid m$ with $r-1\equiv 0\mod p$.}\end{align} In other words $$m^+=\gcd(r+1,m)\quad\und\quad m^-=\gcd(r-1,m).$$

Recall that $n=2^sm$. Let us next consider the divisibility of  $r+1$ and $r-1$ by $2$. Given the possibilities for $\rho(P)$, there are five cases, which we summarize below.
\begin{align}& \vp_2(r-1)\geq s &&\und  &&\vp_2(r+1)\geq s &&\mbox{ if $\rho(P)=1$ with $s\leq 1$},\label{table-r}\\\nn
& \vp_2(r-1)\geq s &&\und  &&\vp_2(r+1)=1 &&\mbox{ if $\rho(P)=1$ with $s\geq 2$},\\\nn
& \vp_2(r-1)=1 &&\und  &&\vp_2(r+1)\geq s &&\mbox{ if $\rho(P)=-1+2^s$ with $s\geq 2$},\\\nn
 & \vp_2(r-1)=1&&\und &&\vp_2(r+1)=s-1 &&
 \mbox{ if $\rho(P)=-1+2^{s-1}$ with $s\geq 3$, and}\\\nn
& \vp_2(r-1)=s-1&&\und &&\vp_2(r+1)=1 &&\mbox{ if $\rho(P)=1+2^{s-1}$ with $s\geq 3$.}
\end{align}
Consequently, letting $$n=n^+n^-\quad\mbox{ with }\quad n^-=\gcd(r-1,n),$$ we discover that
\begin{align}
& n^-=2^sm^-&&\und &&n^+=m^+ &&\mbox{ if $\rho(P)=1$ with $s\geq 0$}\label{tabel-n},\\\nn
& n^-=2m^- &&\und && n^+=2^{s-1}m^+ &&\mbox{ if $\rho(P)=-1+2^s$ or $-1+2^{s-1}$  with $s\geq 2$, and} \\\nn
& n^-=2^{s-1}m^-&&\und && n^+=2m^+ &&\mbox{ if $\rho(P)=1+2^{s-1}$ with $s\geq 3$.}
\end{align} Observe that $n^+\mid r+1$ in all cases, while $n^-$ is even except when $s=0$.
With the above notation in hand, let us now characterize some of the important subgroups of $G$.

\begin{lemma}\label{lemma-techI}
Let $G$ satisfy the General Assumptions for Section \ref{sec-index2}.
Then $$G'=\la \alpha^{r-1}\ra=\la \alpha^{n^-}\ra\quad \und\quad \mathsf Z(G)=\left\{
                                                    \begin{array}{ll}
                                                      G, & \hbox{if $r=1$} \\
                                                      \la \alpha^{n^+}\ra, & \hbox{if $r\neq 1$.}
                                                    \end{array}
                                                  \right.
$$
 In particular, $G$ is non-abelian if and only if $r\neq 1$,
in which case $|G'|=n^+$ and $|\mathsf Z(G)|=n^-$. 
\end{lemma}

\begin{proof}
Let $\tau^{a}\alpha^{x},\,\tau^{b}\alpha^{y}\in G$ be arbitrary elements, where $a,\,b\in \{0,1\}$ and $x,\,y\in [0,n-1]$. Then
\ber \nn[\tau^{a}\alpha^{x},\,\tau^{b}\alpha^{y}]&=&
\alpha^{-x}\tau^{-a}\alpha^{-y}\tau^{-b}\tau^{a}\alpha^{x}\tau^{b}\alpha^{y}\\
&=& \alpha^{-x-r^{a}y+r^{b}x+y}=\alpha^{(r^b-1)x-(r^a-1)y}.\label{com-relat}\eer
Since $r-1$ divides both $r^a-1$ and $r^b-1$, we see from \eqref{com-relat} that all commutator elements live in the subgroup $\la \alpha^{r-1}\ra$. Moreover, taking $a=y=1$ and $x=0$, we see that $\alpha^{r-1}$ is itself a commutator element. This shows that  $G'=\la \alpha^{r-1}\ra$. In particular, $G$ is abelian if and only if $r=1$. Moreover, $\ord(\alpha^{r-1})=\frac{n}{\gcd(r-1,n)}=n^+=\ord(\alpha^{n^-})$, so that $|G'|=n^+$ and $G'=\la \alpha^{r-1}\ra=\la \alpha^{n^-}\ra$  (in view of a finite cyclic group of order $n$ containing a unique subgroup of any given order dividing $n$).

If $r=1$, then $G$ is abelian and $\mathsf Z(G)=G$. Let us next determine $\mathsf Z(G)$ when $r\neq 1$. The element $\tau^{a}\alpha^{x}$ lies in the center of $G$ precisely when \eqref{com-relat} is equal to $1$ for all $b\in \{0,1\}$ and $y\in [0,n-1]$. If $a=1$, then the values $b=0$ and $y=1$ yield a non-identity value in \eqref{com-relat} in view of $r\neq 1$. Therefore $\mathsf Z(G)\leq \la \alpha\ra$. If $a=0$, then taking the value $b=1$ in \eqref{com-relat} shows that only values $x\in [0,n-1]$ with $(r-1)x\equiv 0\mod n$ can correspond to elements of the center. Hence we must have $x\equiv 0\mod n^+$, which means that $\mathsf Z(G)\leq \la \alpha^{n^+}\ra$. However, it is easily seen from \eqref{com-relat} that $\alpha^{n^+}\in \mathsf Z(G)$, whence $\mathsf Z(G)=\la \alpha^{n^+}\ra$. Since $\ord(\alpha^{n^+})=n^-$, we have $|\mathsf Z(G)|=n^-$.
\end{proof}

The following lemma gives a non-cyclic subgroup isomorphic to $C_2\times C_{n^-}$ in most cases, which can then be combined with  Theorem \ref{thm-induc-bound} to bound $\mathsf D(G)$.

\begin{lemma}\label{lemma-techII}
Let $G$ satisfy the General Assumptions for Section \ref{sec-index2}. If $P$ is neither cyclic nor dicyclic,  then
 $$\mathsf C_G(\tau)=\la \alpha^{n^+},\,\tau\ra\cong C_2\times C_{n^-}\quad\mbox{ is non-cyclic}.$$
\end{lemma}

\begin{proof}
Since $P$ is neither cyclic nor dicyclic, Theorem \ref{thm-full-index2-class} shows that  $G$ must have type (A) with $s\geq 1$. In view of  \eqref{tabel-n} and $s\geq 1$,  we have $n^-$ even, whence $C_2\times C_{n^-}$ is non-cyclic.

Let $\tau^a\alpha^x\in G$ be arbitrary, where $a\in \{0,1\}$ and $x\in [0,n-1]$. Then \be\label{comcat} [\tau,\tau^a\alpha^x]=
\tau^{-1}\alpha^{-x}\tau^{-a}\tau\tau^a\alpha^x=\alpha^{-(r-1)x}.
\ee
Now \eqref{comcat} is equal to $1$ precisely when $x\equiv 0\mod n^+$, which means that $\mathsf C_G(\tau)=\la \alpha^{n^+},\,\tau\ra$ with $|\mathsf C_G(\tau)|=2n^-$. In view of Lemma \ref{lemma-techI}, we know $\alpha^{n^+}\in \mathsf Z(G)$, which forces $\mathsf C_G(\tau)=\la \alpha^{n^+},\,\tau\ra$ to be abelian. Consequently, since $\ord(\alpha^{n^+})=n^-$ and $|\mathsf C_G(\tau)|=2n^-$, we conclude that $\mathsf C_G(\tau)$ is isomorphic to either
$C_2\times C_{n^-}$ or $C_{2n^-}$. Thus to complete the proof, we simply need to show that $$\ord(\tau \alpha^{xn^+})<2n^-\quad\mbox{ for all $x\in [0,n^--1]$}.$$ To this end, let $x\in [0,n^--1]$ be arbitrary.
Since $G$ has type (A), we have \be\label{topofwhat}(\tau \alpha^{xn^+})^2=\alpha^{(r+1)xn^+}\quad\mbox{ with } \quad \ord(\tau\alpha^{xn^+})=2\,\ord(\alpha^{(r+1)xn^+}).\ee
Recall that $n^-$ is even (in view of $s\geq 1$),  that $m^+\mid n^+$, that  $m^-\mid n^-$ and  that $m^+m^-=m$ is odd. Thus $$\frac 12 n^-(r+1)xn^+\equiv m^-m^+\equiv m\equiv 0\mod m.$$ As a result, $\ord(\alpha^{(r+1)xn^+})\leq \frac12 n^-$ will follow, proving that $\mathsf C_G(\tau)$ is non-cyclic in view of \eqref{topofwhat}, provided $$\vp_2(\frac12n^-(r+1)n^+)=\vp_2((r+1)n)-1\geq s=\vp_2(n),$$ i.e., provided $\vp_2(r+1)\geq 1$. However, in view of
 \eqref{table-r} and $s\geq 1$, we see that this is indeed the  case, completing the proof.
 \end{proof}

Next, we give the lower bound for $\mathsf D(G)$.

\begin{lemma}\label{index-2-lower}
Let $G$ satisfy the General Assumptions for Section \ref{sec-index2}. Then $$\frac12 |G|+|G'|\leq \mathsf D(G).$$
\end{lemma}

\begin{proof}
From Lemma \ref{lemma-techI}, we know $|G'|=n^+$.
Consider the sequence
$$U=(\tau^{-1}\alpha) \bdot \alpha^{[n^+-1]}\bdot (\tau\alpha^{1-n^+})\bdot \alpha^{[n-1]}\in \Fc(G).$$

Then $|U|=n+n^+=\frac12|G|+|G'|$.
 Since  $(\tau^{-1}\alpha)\alpha^{n^+-1}(\tau\alpha^{1-n^+})\alpha^{n-1}=1$---as is easily seen by recalling from Lemma \ref{lemma-techI} that $\alpha^{n^+}\in \mathsf Z(G)$---it is clear that $U$ is a product-one sequence. Thus to complete the proof, we need to show that $U\in \mathcal A(G)$ is an atom.

Assume to the contrary that we have a factorization $U = V \bdot W$ with
$V,\, W \in \mathcal B (G)$ both nontrivial.
Since $V$ and $W$ are product-one sequences, we have (without restriction) $V \in \mathcal F (\langle \alpha \rangle)$ and
$(\tau^{-1}\alpha) \bdot (\tau\alpha^{1-n^+}) \mid  W$. Hence $V =  \alpha^{[n]}$ and $W =\alpha^{[n^+-2]} \bdot ( \tau^{-1}\alpha) \bdot ( \tau\alpha^{1-n^+})$. Thus there exists a $k \in [0,n^+-2]$ such that
$1 =(\tau^{-1}\alpha)\alpha^k(\tau\alpha^{1-n^+})\alpha^{n^+-2-k} \in  \pi (W)$---in view of Lemma
\ref{permutaton-lemma}, cyclically shifting the terms in a product-one ordered sequence preserves that the sequence has product-one, so we can w.l.o.g. assume our product-one expression  starts with $\tau^{-1}\alpha$. Since $1=(\tau^{-1}\alpha)\alpha^k(\tau\alpha^{1-n^+})\alpha^{n^+-2-k}=
\alpha^{(r-1)(k+1)}$,  it follows that $k+1\in [1,n^+-1]$ must be a multiple of $\ord(\alpha^{r-1})$. However, since $n^-=\gcd(r-1,n)$ with $n=n^+n^-$, it follows that $\ord(\alpha^{r-1})=n^+$, so that  $k+1\in [1,n^+-1]$ cannot be  a multiple of $\ord(\alpha^{r-1})$.
This contradiction establishes the desired lower bound for $\mathsf D(G)$.
\end{proof}

The next lemma reduces the problem of finding a matching upper bound for $\mathsf D(G)$ to the case when $|G'|=n^+$ is prime.

\begin{lemma}\label{lemma-reduce-to-prime}
Let $G$ satisfy the General Assumptions for Section \ref{sec-index2}. Suppose $G$ is non-abelian, let $p$ be a prime divisor of $|G'|=n^+$, and let $$H=\la \alpha^{\frac{n^+}{p}},\, \tau\ra\leq G.$$
Then $H$ has a cyclic, index $2$ subgroup and $|H'|=p$, where $H'=[H,H]\leq H$ is the commutator subgroup of $H$. In particular, if $\mathsf D(H)\leq \frac12 |H|+|H'|$, then $\mathsf D(G)\leq \frac12|G|+|G'|$.
\end{lemma}

\begin{proof}
Observe that $\ord(\alpha^{\frac{n^+}{p}})=\ord(\alpha)\frac{p}{n^+}=n^-p$. If $G$ has type (A), then $\tau^2=1\in \la \alpha^{\frac{n^+}{p}}\ra$. If $G$ has type (B), then $s\geq 1$ and $n^-$ is even. Thus  $\tau^2=\alpha^{\frac{n}{2}}\in \la \alpha^{\frac{n^+}{p}}\ra$ since $(\frac{n^+}{p})\frac{pn^-}{2}=\frac{n}{2}$ with $2\mid n^-$.
If $G$ has type (C), then $P$ is cyclic. Hence \eqref{tabel-n} implies that  $m^+=n^+$, and now $\tau^2=\alpha^m\in \la \alpha^{\frac{n^+}{p}}\ra$ holds in view of  $\left(\frac{n^+}{p}\right) m^-p=\frac{m^+m^-p}{p}=m$.
In all cases, we  conclude that $$|H|=2\,\ord(\alpha^{\frac{n^+}{p}})=2n^-p,$$ so that $\la \alpha^{\frac{n^+}{p}}\ra\leq H$  is a cyclic, index $2$ subgroup.

Next, let us compute $H'\leq H$.
Let $\tau^{a}\alpha^{x},\,\tau^{b}\alpha^{y}\in H$ be arbitrary elements, where $a,\,b\in \{0,1\}$,  $x,\,y\in [0,n-1]$ and $x\equiv y\equiv 0\mod \frac{n^+}{p}$. Then (as in Lemma \ref{lemma-techI})
\ber [\tau^{a}\alpha^{x},\,\tau^{b}\alpha^{y}]=\alpha^{(r^b-1)x-(r^a-1)y}.
\label{com-relatIIh}\eer
Since $x\equiv y\equiv 0\mod \frac{n^+}{p}$ and since $r-1$ divides both $r^b-1$ and $r^a-1$, we see from \eqref{com-relatIIh} that all commutator elements live in the subgroup $\la \alpha^{(r-1)\frac{n^+}{p}}\ra$. Moreover, taking $a=1$, $y=\frac{n^+}{p}$ and $x=0$, we see that $\alpha^{(r-1)\frac{n^+}{p}}$ is itself a commutator element. This shows that  $H'=\la \alpha^{(r-1)\frac{n^+}{p}}\ra$. In consequence, since $\gcd(r-1,n)=n^-$ and $n=n^+n^-$, it follows that $|H'|=p$.

Now $|H|=2n^-p$, \, $|H'|=p$, \ $|G'|=n^+$ (from Lemma \ref{lemma-techI}) and $|G:H|=\frac{2n}{2n^-p}=\frac{n^+}{p}$. Thus, if $\mathsf D(H)\leq \frac12 |H|+|H'|$, then Theorem \ref{thm-induc-bound} yields
\[
\mathsf D(G)\leq \mathsf D(H)|G:H|\leq (\frac12|H|+|H'|)|G:H|=\frac12|G|+|H'||G:H|=n+n^+=\frac12|G|+|G'| \,. \qedhere
\]
\end{proof}

The following lemma handles the case when there are a sufficient number of terms from $\la \alpha\ra$.

\begin{lemma}\label{lemma-lotsofalplha}
Let $G$ satisfy the General Assumptions for Section \ref{sec-index2}. Suppose  $n^+=p$ is prime and let $U\in \Fc(G)$ be a product-one sequence. If  $|U|\geq n+p+1$ and $U$ contains at least $p-1$ terms from $\la \alpha\ra\setminus \mathsf Z(G)$, then $U$ is not an atom.
\end{lemma}

\begin{proof}
Since $n^+=p$ is prime, we have $n^+=p\geq 2$. Thus Lemma \ref{lemma-techI} implies that $G$ is non-abelian with  $\mathsf Z(G)=\la \alpha^{n^+}\ra=\la \alpha^p\ra$ and $G'=\la \alpha^{n^-}\ra=\la \alpha^{r-1}\ra$. In particular, $|G'|=n^+=p\geq 2$ and $|\mathsf Z(G)|=n^-$.

By hypothesis, there is a subsequence $V\mid U$ with $\supp(V)\subset \la \alpha\ra\setminus \mathsf Z(G)$ and $|V|= p-1$, say $V= v_1\bdot\ldots\bdot  v_{p-1}$ with $$v_i=\alpha^{x_i}\quad\mbox{ for } \quad i\in [1,p-1],$$  where $x_i\in [0,n-1]$. Since $v_i\notin \mathsf Z(G)=\la \alpha^p\ra$ for all $i\in [1,p-1]$,  we see that \be\label{not0pII}x_i\not\equiv 0\mod p\quad\mbox{ for all $i\in [1,p-1]$}.\ee

If $\supp(U)\subset \la \alpha\ra$, then $|U|\geq n+p+1>n=|\la \alpha\ra|\geq \mathsf D(\la \alpha\ra)$ ensures that $U$ cannot be an atom, as desired, where the final inequality follows from Lemma \ref{dav-up-lemma}. Therefore we can assume
 there is some  $z= \tau\alpha^x\in \supp(U)$ with $x\in [0,n-1]$.

As remarked in Section
\ref{sec-notation}, $\pi(V\bdot z)$ is contained in a $G'$-coset. Let us next show that \be\label{combigI}|\pi(V\bdot z)|=p=|G'|,\ee so that $\pi(V\bdot z)$ is an entire $G'$-coset.

Let ${W^*}$ be an ordering of the terms of $V\bdot z$, so $W^*\in \Fc^*(G)$ with $[W^*]=V\bdot z$. Then \be\label{exponentsI}\pi(W^*)=\tau\alpha^{x+\Sum{i=1}{p-1}\epsilon_ix_i},\ee where $\epsilon_i=1$ if the term $x_i$ occurs to the right of $z=\tau\alpha^x$ in $W^*$, and $\epsilon_i =r$ if the term $x_i$ occurs to the left of $z=\tau\alpha^x$ in $W^*$. The possible exponents for $\alpha$ in \eqref{exponentsI} (as we range over all possible orderings $W^*$ of $V\bdot z$) are then $$x+\{x_1,rx_1\}+\ldots+\{x_{p-1},rx_{p-1}\}=x+
\Sum{i=1}{p-1}x_i+\{0,(r-1)x_1\}+\ldots+\{0,(r-1)x_{p-1}\}.$$
Consequently, \be\label{siidI}\pi(V\bdot z)=\tau\alpha^{x+\Sum{i=1}{p-1}x_i}\{(\alpha^{r-1})^y:\; y\in Y\},\ee where $Y=\{0,x_1\}+\ldots+\{0,x_{p-1}\}$. Recall that $\alpha^{r-1}$ is a generator for $G'$ having $\ord(\alpha^{r-1})=n^+=p$. Thus the cardinality of $\pi(V\bdot z)$ is just the number of residue classes modulo $p$ in $Y=\{0,x_1\}+\ldots+\{0,x_{p-1}\}$.
From \eqref{not0pII}, we see that each set $\{0,x_i\}$ consists of $2$ elements that are distinct modulo $p$, in which case applying the Cauchy-Davenport Theorem to $Y$ shows that $|Y|=p=|G'|$, which combined with \eqref{siidI} establishes \eqref{combigI}, as claimed.

Now $|U\bdot (V\bdot z)^{-1}|=|U|-p\geq n+1=\dd(G)+1$, with the first inequality by hypothesis and the final equality from Theorem \ref{thm-olsonwhite-upper}.
Thus we can apply the definition of $\mathsf d(G)$ to $U\bdot (V\bdot z)^{-1}$ to find a nontrivial, product-one subsequence $T\mid U\bdot (V\bdot z)^{-1}$.
But now Lemma \ref{lemma-G'} shows that  $\pi(U\bdot T^{[-1]})\subset G'$.
As a result, since $V\bdot z\mid U\bdot T^{[-1]}$ follows from the definition of $T$, it follows in view of \eqref{combigI} that $\pi(U\bdot T^{[-1]})=G'$. In particular, $1\in G'= \pi(U\bdot T^{[-1]})$. Thus $U=(U\bdot T^{[-1]})\bdot T$ is a factorization of $U$ into two nontrivial, product-one subsequences, ensuring that $U$ is not an atom, as desired. \end{proof}

When either $n^+$ or $n^-$ is too small, the general strategy for proving Theorem \ref{thm-index2} breaks down, requiring the cases when $n^-\leq 2$ or $n^+\leq 2$ to be handled separately. Most of these remaining cases can be handled by simple arguments. However, the case when $G$ is isomorphic to a dicyclic group $Q_{4p}$ with $p$ odd is particularly difficult, so we handle it separately now.

\begin{lemma}\label{lemma-dicyclic}
Let $G$ be a dicyclic group of order $4p$ with $p$ an odd prime, say
$$G = Q_{4p}=\la \alpha,\,\tau\mid \alpha^{2p}=1,\quad \tau^2=\alpha^{p},\quad \alpha\tau=\tau\alpha^{-1}\ra.$$ Then $\mathsf D(G)\leq \frac12|G|+|G'|=3p$, where $G'=[G,G]\leq G$ is the commutator subgroup.
\end{lemma}

\begin{proof}
By hypothesis,   $G$ satisfies the Standard Assumptions of Section \ref{sec-index2} having types (B) and (C) (since these types coincide for $s=1$) with $$s=1, \quad P\leq G\;\mbox{ cyclic},\quad n=2p, \quad r=2p-1, \quad n^-=2, \quad m^-=1,\quad\und\quad n^+=m^+=p.$$ As a result, Lemma \ref{lemma-techI} tells us that \be\label{com-center}G'=\la \alpha^2\ra\cong C_p\quad\und\quad \mathsf Z(G)=\la \alpha^p\ra\cong C_2.\ee

Assume by contradiction that we have some atom $U\in \mathcal A(G)$ with $|U|=\mathsf D(G)\geq 3p+1$. Since $U$ is a product-one sequence, there is an ordering of its terms with product $1$, say $U^*\in \Fc^*(G)$ with $[U^*]=U$ and $\pi(U^*)=1$.

\medskip

Suppose $\mathsf Z(G)\cap \supp(U)\neq \emptyset$. Since $U$ is an atom and $G$ is nontrivial, we cannot have $1\in \supp(U)$. Thus, in view of \eqref{com-center} and $\mathsf Z(G)\cap \supp(U)\neq \emptyset$, we must have $\alpha^p\in \supp(U)$. By Lemma \ref{permutaton-lemma}, we can w.l.o.g. assume $\alpha^p$ is equal to the first term of $U^*$, so $U^*(1)=\alpha^p$. But then $|U^*(2,|U|-1)|=|U|-2\geq 3p-1\geq 2p=|G/\mathsf Z(G)|$, which means we can apply Lemma \ref{dav-up-lemma} to $\phi_{\mathsf Z(G)}\Big(U^*(2,|U|-1)\Big)$ and thereby find a nontrivial, consecutive subsequence of $U^*(2,|U|-1)$ with product from $\mathsf Z(G)=\{1,\alpha^p\}$, say $U^*(I)$ with $I\subset [2,|U|-1]$ an interval. Since $U=[U^*]\in \mathcal A(G)$ is an atom, Lemma \ref{lemma-no-consecutive-prod-1} ensures that $\pi\Big(U^*(I)\Big)\neq 1$.
Thus $\pi\Big(U^*(I)\Big)=\alpha^p\in \mathsf Z(G)$ with $U^*(I)\mid U^*(2,|U|-1)$ consecutive, in which case  $\pi\Big(U^*(1)\bdot U^*(I)\bdot U^*([2,|U|]\setminus I)\Big)=\pi(U^*)=1$.  However, $\pi\Big(U^*(1)\bdot U^*(I)\Big)=\alpha^p\alpha^p=\alpha^{2p}=1$, so that $U=[U^*]=[U^*(1)\bdot U^*(I)]\bdot [U^*([2,|U|]\setminus I)]$ is a factorization of $U$ into $2$ nontrivial, product-one subsequences---the subsequence $[U^*([2,|U|]\setminus I)]$ is nontrivial since $I\subset [2,|U|-1]$---contradicting that $U\in \mathcal A(G)$ is an atom in this case as well. So we instead conclude that \be \mathsf Z(G)\cap \supp(U)= \emptyset\label{nocenterterms}.\ee
In view of Lemma \ref{lemma-lotsofalplha} and \eqref{nocenterterms}, we may assume \be\label{wahaa}\mbox{there are at most $p-2$ terms of $U$ from $\la \alpha\ra$}.\ee

\bigskip

Let $J\subset [1,|U|]$ be all those indices $j\in [1,|U|]$ with $U^*(j)\in \tau\la \alpha\ra$. Since $\pi(U^*)=1$, it is easily deduced from the group presentation for $G$ that $|J|$ must be even. In view of \eqref{wahaa}, we have $|J|\geq |U|-p+2\geq 2p+3$. Thus, since $|J|$ must be even, it follows that \be\label{Jbigboy} |J|\geq 2p+4.\ee Let $$j_1<j_2<\ldots <j_{2w-1}<j_{2w}$$ be the distinct elements of $J$, where $$w=\frac12 |J|\geq p+2.$$ In view of Lemma \ref{permutaton-lemma}, we can cyclically shift the ordering $U^*$ of $U$ until the first term of $U^*$ is from $\tau\la \alpha\ra$, i.e., such that $j_1=1$.

Now define an ordered sequence
\[
{U'}^*= U^*(j_1,j_2-1)\bdot U^*(j_2,j_3-1) \bdot\ldots\bdot  U^*(j_{2w-1},j_{2w}-1)\bdot U^*(j_{2w},|U|)\in \Fc^*(G) \,.
\]
The ordered sequence ${U'}^*$ is obtained from the  product-one ordered sequence $U^*$ by repeatedly replacing a consecutive subsequence with a single term equal to its product. As noted in Section \ref{sec-notation}, since $[U^*]=U\in \mathcal A(G)$ was an atom, this ensures that $$U':=[{U'}^*]\in \mathcal A(G)$$ is also an atom.
From the definition of the $j_i$, each $U^*(j_i,j_{i+1}-1)$, for $i\in [1,2w]$ where $j_{2w+1}=|U|+1$, has its first term from $\tau\la \alpha\ra$ and all other terms from $\la \alpha\ra$.
In consequence, we have $$\supp(U')\subset \tau\la \alpha\ra\quad\und\quad |U'|=|J|=2w\geq 2p+4,$$ where the inequality follows from \eqref{Jbigboy}.

Define a map $\overline{\bdot}:\tau\la \alpha\ra\rightarrow \Z/2p\Z$ by setting $\overline{\tau\alpha^x}:=\phi_{2p\Z}(x)\in \Z/2p\Z$, i.e., $\tau\alpha^x$ maps to the residue class represented by $x$ modulo $p$. Since $\ord(\alpha)=2p$, the map $\overline{\bdot}$ is well-defined.
We continue with a straightforward claim.

\subsection*{Claim A} Let  $R\in \Fc(G)$ with $\supp(R)\subset \tau\la \alpha\ra$. Then $R$ is a product-one sequence precisely when there exists a factorization $R=R^+\bdot R^-$ such that $|\overline{R^+}|=|\overline{R^-}|$ and $\sigma(\overline {R^-})-\sigma(\overline{R^+})=\frac12|R|p$.

\begin{proof}
Suppose $R$ is a product-one sequence. Then there exists an ordering of $R$, say $R^*\in \Fc^*(G)$ with $[R^*]=R$, such that $\pi(R^*)=1$.  Since $\pi(R^*)=1$ and $\supp(R)\subset \tau\la \alpha\ra$, it is easily deduced from the group presentation for $G$ that $|R|$ must be even. Thus let $$R^*=( \tau\alpha^{r^-_1}) \bdot ( \tau\alpha^{r^+_1})\bdot\ldots\bdot ( \tau\alpha^{r^-_w}) \bdot (\tau\alpha^{r^+_w}) \in \Fc^*(G),$$ where $r^-_i,\,r^+_i\in [0,2p-1]$ and $w=\frac12|R|$. Repeatedly applying the group presentation relations for $G$ yields
\be\label{nolabelsmiley}1=\pi(R^*)=(\tau\alpha^{r^-_1}\tau\alpha^{r^+_1})(\tau\alpha^{r^-_2}
\tau\alpha^{r^+_2})\ldots (\tau\alpha^{r^-_w}\tau\alpha^{r^+_w})=\alpha^{wp+\Sum{i=1}{w}r_i^+-\Sum{i=1}{w}r_i^-},\ee
thus implying \be\label{starkk}wp+\Sum{i=1}{w}r_i^+-\Sum{i=1}{w}r_i^-\equiv 0\mod 2p.\ee
Let $$R^-=[R^*(I^-)]= (\tau\alpha^{r^-_1}) \bdot\ldots\bdot ( \tau\alpha^{r^-_w})\quad\und\quad R^+=[R^*(I^+)]= (\tau\alpha^{r^+_1})\bdot\ldots\bdot ( \tau\alpha^{r^+_w}),$$ where $I^-=\{1,3,\ldots,2w-1\}$ and $I^+=\{2,4,\ldots,2w\}$. Since $I^-\cup I^+=[1,2w]=[1,|R|]$ with the union disjoint, we see that $R=R^+\bdot R^-$ with $|\overline{R^-}|=|R^-|=|R^+|=|\overline{R^+}|=w=\frac12|R|$. Moreover, \eqref{starkk} is equivalent to saying $\sigma(\overline{R^-})-\sigma(\overline{R^+})=wp=\frac12|R|p$. Thus one direction of the claim in established.

Now suppose that we have a  factorization $R=R^+\bdot R^-$ such that $|\overline{R^+}|=|\overline{R^-}|$ and $\sigma(\overline {R^-})-\sigma(\overline{R^+})=\frac12|R|p$. Let $R^*\in \Fc^*(G)$ be an ordering of $R$ such that $[R^*(I^-)]=R^-$ and $[R^*(I^+)]=R^+$, where $I^-\subset [1,|R|]$ is the subset of odd indices and $I^+\subset [1,|R|]$ is the subset of even indices. Since $R=R^+\bdot R^-$ with $|R^+|=|\overline{R^+}|=|\overline{R^-}|=|R^-|$, it follows that $|R|$ is even, so that $|I^+|=|I^-|=\frac12|R|$.
Let $w=\frac12|R|$ and let $$R^*(2i-1)=\tau\alpha^{r^-_i}\quad\und\quad R^*(2i)=\tau\alpha^{r^+_i}\quad\mbox{ for $i\in [1,\frac12|R|]=[1,w]$}.$$ Then, in view of $\sigma(\overline {R^-})-\sigma(\overline{R^+})=\frac12|R|p=wp$, we see that \eqref{starkk} holds, and consequently also \eqref{nolabelsmiley}. Thus $1=\pi(R^*)\in \pi([R^*])=\pi(R)$, showing that $R$ is a product-one sequence, which completes the claim.
\end{proof}

Using Claim A, we see that Lemma \ref{lem-factorizationI} is equivalent to saying that the maximal length of an atom $V\in \mathcal A(G)$ with $\supp(V)\subset \tau\la \alpha\ra$ is $|V|\leq 2p+3$. However, this contradicts that  we constructed above an atom $U'\in \mathcal A(G)$ with  $\supp(U')\subset \tau\la \alpha\ra$ and $|U'|\geq 2p+4$, completing the proof.
\end{proof}

With the above preparatory work complete, we are now ready to begin the proof of Theorem \ref{thm-index2}

\smallskip
\begin{proof}[\bf Proof of Theorem \ref{thm-index2}]
If $G$ is cyclic, then $\dd(G)=|G|-1$, while $\dd(G)=\frac12|G|$ follows for non-cyclic $G$ having a cyclic, index $2$ subgroup, and
if $G$ is abelian, then $\mathsf D(G)=\dd(G)+1$  (by Lemma \ref{dav-up-lemma} and  Theorem \ref{thm-olsonwhite-upper}).

Therefore we may assume $G$ is non-abelian and satisfies the General Assumptions for Section \ref{sec-index2}. Lemma \ref{index-2-lower} gives $\dd(G)+|G'|=\frac12|G|+|G'|\leq \mathsf D(G)$.
Since $G$ is non-abelian, Lemma \ref{lemma-techI} gives $|G'|=n^+\geq 2$, and it remains to show the upper bound \be\label{targett}\mathsf D(G)\leq \frac12|G|+|G'|=n+n^+.\ee

By Lemma \ref{lemma-reduce-to-prime}, it suffices to prove \eqref{targett} when $|G'|=n^+=p$ is prime. Furthermore, if $n^+=2$, then Lemma \ref{lemma-|G'|=2} yields \eqref{targett}. Consequently, we can assume \be\label{n^+assump}|G'|=n^+=p\geq 3\quad\mbox{ is prime}.\ee
In particular, only the cases where $n^+=p$ is odd remain, which in view of \eqref{tabel-n} means that $\rho(P)=1$. From the definition  of $\rho$, we see that $\rho(P)=1$ corresponds to when $P\cong C_{2^{s+1}}$ or $P\cong C_2\times C_{2^s}$. However, if $P\cong C_2\times C_{2^s}$ is non-cyclic, then Lemma \ref{lemma-techII} shows that $\mathsf C_G(\tau)\cong C_2\times C_{n^-}$ is non-cyclic. Since $\mathsf D(C_2\times C_{n^-})=n^-+1$ is well-known (\cite[Theorem 5.8.3]{Ge-HK06a}), invoking Theorem \ref{thm-induc-bound} would then yield $$\mathsf D(G)\leq \mathsf D(\mathsf C_G(\tau))|G:\mathsf C_G(\tau)|=\mathsf D(C_2\times C_{n^-})\,n^+=(n^-+1)n^+=n+n^+,$$ yielding \eqref{targett}. So it remains to prove \eqref{targett} when $$P\cong C_{2^{s+1}}\quad\mbox{ is cyclic with $\rho(P)=1$}.$$ In particular, Theorem \ref{thm-full-index2-class} now tells us that $G$ has type (C).

\medskip

If $n^-=1$, then \eqref{tabel-n} and the definition of $n^-$ and $m^-$ ensure that $s=0$, \ $r=n-1$ and $p=n^+=n$. This corresponds to the case when  $G$ is dihedral of order $2n$ with $n$ odd. In this case, Lemma \ref{dav-up-lemma} implies $\mathsf D(G)\leq |G|=2n=n+n^+$, yielding \eqref{targett}. Therefore we may assume $n^-\geq 2$.

Suppose $n^-=2$. Then it follows in view of  $\rho(P)=1$ and \eqref{tabel-n} that $$s=1, \quad m^-=1, \quad n^+=m^+=m=p\quad\und\quad n=2m=2p.$$
Since $P$ is cyclic with $s=1$, Theorem \ref{thm-full-index2-class} ensures that $G$ has types (C) and (B) (these types coincide for $s=1$) with $$(r-1)(r+1)=r^2-1\equiv 0\mod m\quad\und\quad r\equiv 1\mod 2.$$ In consequence, since $1=m^-=\gcd(r-1,m)$ and $r\in [1,n]$, it follows that $r=n-1=2p-1$. As a result, we see that $G\cong Q_{4p}$ is dicyclic, in which case Lemma \ref{lemma-dicyclic} yields \eqref{targett}. So we may assume \be\label{n^-assump} n^-\geq 3.\ee

\medskip

To establish \eqref{targett}, assume by contradiction that we have an atom $U\in \mathcal A(G)$ with \be\label{wickergoal}|U|=\mathsf D(G)\geq n+n^++1=n^+n^-+n^++1.\ee Factor $U=U_\alpha \bdot U_\tau$ with $\supp(U_\alpha)\subset \la \alpha\ra$ and $\supp(U_\tau)\subset \tau\la \alpha\ra$.  In view of Lemma \ref{lemma-techI}, we know $$\mathsf Z(G)=\la \alpha^{n^+}\ra=\la \alpha^p\ra\quad\und\quad G'=\la \alpha^{n^-}\ra=\la \alpha^{r-1}\ra\quad\mbox{ with }\quad |G'|=n^+=p.$$ Let $U'_\alpha\mid U_\alpha$ be the subsequence consisting of all terms from $\la \alpha\ra\setminus \mathsf Z(G)$. Then, since $\mathsf Z(G)=\la \alpha^{p}\ra$, we see that $U_\alpha\bdot {U'_\alpha}^{[-1]}$ is the subsequence of $U$ consisting of all terms from $\mathsf Z(G)$.

\medskip

Let us next show that \be\label{Vsmall}|U'_\alpha|\leq n^+-2,\quad |U_\alpha \bdot {U'_\alpha}^{[-1]}|\leq n^--1\quad\und\quad |U_\alpha|\leq n^++n^--3.\ee
In view of Lemma \ref{lemma-lotsofalplha}, we have $|U'_\alpha|\leq n^+-2$. Thus, if \eqref{Vsmall} fails, then we must have
$|U_\alpha\bdot {U'_\alpha}^{[-1]}|=|U_\alpha|-|U'_\alpha|\geq n^-$. In other words, there are at least $n^-=|\mathsf Z(G)|$ terms of $U$ from $\mathsf Z(G)$.
Since $U\in \mathcal A(G)$ is an atom, let $U^*\in \Fc^*(G)$ with $[U^*]=U$ be an ordering of $U$ such that $\pi(U^*)=1$. Any term from $\mathsf Z(G)$ can be moved around in the ordered sequence $U^*$ without changing the value of $\pi(U^*)$. Thus we can w.l.o.g. assume all terms from $U_\alpha \bdot {U'_\alpha}^{[-1]}$ are consecutive in $U^*$.
In consequence, since $|U_\alpha\bdot {U'_\alpha}^{[-1]}|\geq |\mathsf Z(G)|=n^-$,  we can apply Lemma \ref{dav-up-lemma} to $U_\alpha \bdot {U'_\alpha}^{[-1]}$ to find a nontrivial, consecutive, product-one subsequence $U^*(I)$, where $I\subset [1,|U|]$ is an interval.  Moreover, $|U^*(I)|\leq |\mathsf Z(G)|=n^-<|U|$, meaning $U^*(I)\mid U$ is proper. But since $U=[U^*]\in \mathcal A(G)$ is an atom, this contradicts Lemma \ref{lemma-no-consecutive-prod-1}. So \eqref{Vsmall} is established, as claimed.

\bigskip

Define a map $\iota:G\rightarrow \Z$ by setting $$\iota(\tau^y\alpha^x)=x,\quad\mbox{ where $x\in [0,n-1]$ and $y\in [0,1]$},$$ and define a map $\overline{\bdot}:G\rightarrow \Z/p\Z$ by setting $$\overline{g}=\phi_{p\Z}(\iota(g))\quad\mbox{ for $g\in G$},$$ so $g=\tau^y\alpha^x\in G$ maps to the residue class modulo $p$ given by $\iota(g)=x$.

\medskip

Let $R\in \Fc(G)$ be a sequence and let  $R^*\in \Fc^*(G)$ be an arbitrary ordering of $R$, so $[R^*]=R$.  Factor $R=R_\alpha \cdot R_\tau$ with $\supp(R_\alpha)\subset \la \alpha\ra$ and $\supp(R_\tau)\subset \tau \la \alpha\ra$.
We proceed to describe $\pi(R)$ under the assumption that $$|R_\tau|\geq 1.$$
  First note that, from the defining relations for $G$, it is clear that $\pi(R^*)\in \la \alpha\ra$ if and only if the number of terms of $R$ from $\tau\la \alpha\ra$ is even, that is, if $|R_\tau|$ is even. Let $\omega=\lfloor \frac12 |R_\tau|\rfloor$, so that $|R_\tau|=2\omega$ when $\pi(R^*)\in \la \alpha\ra$ and $|R_\tau|=2\omega+1$ when $\pi(R^*)\in \tau\la \alpha\ra$.

Next, since $G$ has type (C), a routine application of the defining relations for $G$ shows that  \be\label{relationtime} \pi(R^*)=\tau^\epsilon\alpha^{\omega m+\Sum{i=1}{|R|}d_i\iota(R^*(i))},\ee
where
$\epsilon=1$ if $|R_\tau|$ is odd, $\epsilon =0$ if $|R_\tau|$ is even,
$d_i=1$ if the number of terms of $R^*$ from $\tau\la \alpha\ra$ to the right of $R^*(i)$ is even, and $d_i=r$ if the number of terms of $R^*$ from $\tau\la \alpha\ra$ to the right of $R^*(i)$ is odd.

There are some important consequences of the formula \eqref{relationtime}.  Let $I\subset [1,|R|]$ be the set of indices such that $[R^*(I)]=R_\tau$.
If we fix the position of every term  $R^*(i)\in\la \alpha\ra$ with $i\notin I$ but allow ourselves to permute the terms within $R^*(I)$, this maintains that $[R^*(I)]=R_\tau$ while  each coefficient $d_i$, for $i\in [1,|R|]\setminus I$, remains unaffected and constant. In consequence, when trying to determine the possible values for \eqref{relationtime} over all orderings $R^*$, we can first decide how to distribute the terms from $R_\alpha$ into $R^*$, thus fixing and determining the subset of indices $I\subset [1,|R|]$ with $[R^*(I)]=R_\tau$,  and then decide how to permute the terms within $R^*(I)$.
 Since $|R_\tau|\geq 1$,
every term of $R_\alpha$ can either be placed in $R^*$ such that the number of terms of $R^*$ from $\tau\la \alpha\ra$ to its right is even, or such that this number is odd. Changing this choice has the effect on \eqref{relationtime} of switching $d_i$ between $1$ and $r$. Once we have fixed how the terms of $R$ from $\la \alpha\ra$ are to be distributed in $R^*$, the set $I\subset [1,|R|]$ is then fixed, but we are free to re-order the terms from $R_\tau$ so long as we preserve $[R^*(I)]=R_\tau$ and this will not affect whether $d_i=1$ or $d_i=r$ holds for any $i\in [1,|R|]\setminus I$.

Concerning the terms of $R^*$ from $R_\tau$, whether $d_i=1$ or $d_i=r$ holds for $i\in I$ depends entirely on whether $R^*(i)=(R^*(I))(j)$ with \ $j\equiv \epsilon\mod 2$ \ or \ $j\equiv \epsilon -1\mod 2$. \ If $j\equiv \epsilon\mod 2$, then $d_i=1$, and if $j\equiv \epsilon-1\mod 2$, then $d_i=r$.  Letting $$J=\{1+\epsilon,3+\epsilon,\ldots,2\omega-1+\epsilon\}\subset [1,|R_\tau|]$$ be the subset of  indices congruent to $\epsilon -1$ modulo $2$, we are free to arrange for $[(R^*(I))(J)]$ to be  any subsequence of $R_\tau$ having length $\omega=\lfloor\frac12|R_\tau|\rfloor$, and then $d_i=r$ will hold for all these terms, while $d_i=1$ will hold for all remaining terms of $R_\tau$.

In summary, the above works shows that
$$\pi(R)=\tau^\epsilon\alpha^{\omega m}\{\alpha^{x}:\; x\in X\},$$ where \ber\nn X&=&\Sum{i=1}{|R_\alpha|}\left\{\iota\Big(R_\alpha^*(i)\Big),\;r\,\iota\Big(R_\alpha^*(i)\Big)\right\}
+
\left\{r\,\sigma(\iota(R_\tau'))+\sigma(\iota(R_\tau\bdot {R_\tau'}^{[-1]})):\;R_\tau'\mid R_\tau,\; |R_\tau'|=\omega=\left\lfloor\frac12|R_\tau|\right\rfloor\right\}\\
&=& \sigma(\iota(R))+(r-1)\Big(\left\{0,\iota\big(R_\alpha^*(1)\big)\right\}+\ldots+
\left\{0,\iota\big(R_\alpha^*(|R_\alpha|)\big)\right\}+
\Sigma_{\left\lfloor\frac12|R_\tau|\right\rfloor}(\iota(R_\tau))\Big) \nn
\eer
 and $R^*_\alpha\in \Fc^*(G)$ is any ordering of $R_\alpha$
Consequently, \be\label{Rpi}\pi(R)=\tau^{\epsilon}\alpha^{xm+\sigma(\iota(R))}\{(\alpha^{r-1})^y:\; y\in Y\},\ee where $$Y=\left\{0,\iota\Big(R_\alpha^*(1)\Big)\right\}+\ldots+
\left\{0,\iota\Big(R_\alpha^*(|R_\alpha|)\Big)\right\}
+\Sigma_{\left\lfloor\frac12|R_\tau|\right\rfloor}(\iota(R_\tau)) \,.$$ Since $\ord(\alpha^{r-1})=\ord(\alpha^{n^-})=n^+=p$, we conclude that $|\pi(R)|$ is equal to the number of distinct residue classes modulo $p$ in  $Y$.

\bigskip

Let us next apply some of the above reasoning to the sequence $U$ in the following claim, which shows that any sufficiently small subsequence can be placed in an ordering of $U$ with product one so as to avoid some long length, consecutive subsequence.

\subsection*{Claim A} If $T\mid U$ is a subsequence with $|T|\leq n^+-1$, then there exists an ordering of $U$, say $U^*\in \Fc^*(G)$ with $[U^*]=U$, and an interval $J\subset [1,|U|]$ such that $\pi(U^*)=1$, \ $T\mid \left[U^*\Big([1,|U|]\setminus J\Big)\right]$ and $|J|\geq 2n^-$.

\begin{proof}
Since $U\in \mathcal A(G)$ is an atom, there is an ordering of $U$, say $U^*\in \Fc^*(G)$ with $[U^*]=U$, such that  $\pi(U^*)=1$. In view of \eqref{Vsmall} and \eqref{wickergoal}, we know $\supp(U)\cap \tau\la \alpha\ra\neq \emptyset$. Thus, in view of  Lemma \ref{permutaton-lemma}, we can cyclically shift the terms of $U^*$ until w.l.o.g. $U^*(1)\in \tau\la \alpha\ra$. In view of the formula \eqref{relationtime} for $U^*=R^*$, we see that we can shift the position of a term $x$ of $U^*$ from $\la \alpha\ra$ while preserving that $\pi(U^*)=1$ so long as we maintain the parity of the number of terms of $U^*$ from $\tau\la \alpha\ra$ that follow to the right of $x$. In particular, we can put all terms of $U^*$ from $\la \alpha\ra$ for which this number is odd into a consecutive block starting with the second term of $U^*$, while also putting all terms of $U^*$ from $\la \alpha\ra$ for which this number is even into a consecutive block at the very end of $U^*$, and this will preserve that $\pi(U^*)=1$. In other words, we may w.l.o.g. assume $U^*$ has the form
\[
U^*= (\tau\alpha^{x_1})\bdot ( \alpha^{y_1}\bdot\ldots \bdot \alpha^{y_{t}})\bdot ( \tau\alpha^{x_2}\bdot\ldots\bdot  \tau\alpha^{x_{2w}})\bdot( \alpha^{y'_1}\bdot\ldots\bdot \alpha^{y'_{t'}}) \,,
\]
for some $t,\,t'\geq 0$ with $$t+t'=|U_\alpha|,\quad 2w=|U_\tau|\geq 2,\quad\und\quad x_i,\,y_i,\,y'_i\in [0,n-1].$$
Let $J'\subset [1,|U|]$ be those indices $j\in [1,|U|]$ such that $U^*(j)\in  \tau\la \alpha\ra$, i.e.,
$$J'=\{1\}\cup [t+2,t+2w].$$

Now $T=U^*(I)$ for some $I\subset [1,|U|]$. Factor $T=T_\alpha \bdot T_\tau$ with $\supp(T_\alpha)\subset \la \alpha\ra$  and $\supp(T_\tau)\subset \tau\la \alpha\ra$. Since $\supp(U^*(J'))\subset \tau\la \alpha\ra$, we see that $T_\alpha$ is disjoint from  $U^*(J')$. For the remaining terms of $T$, we must have $T_\tau=(U^*(J'))(X)$ for some subset $X\subset [1,2w]$.
Let $X=X^+\cup X^-$, where $X^+\subset X$ is the subset of
 indices $x\in X$ with $x$ even and $X^-\subset X$ is the subset of  indices $x\in X$ with $x$ odd.
Consider an arbitrary term of $T$ from $\tau\la \alpha\ra$, say $(U^*(J'))(x)$ with $x\in X\subset [1,2w]$. If $x\in X^-$, then $(U^*(J'))(x)$ can be moved freely about in
$(U^*(J'))(\{1,2,\ldots,2w-1\})$ without changing that $\pi(U^*)=1$. Likewise, if $x\in X^+$, then $(U^*(J'))(x)$ can be moved freely about in $(U^*(J'))(\{2,4,\ldots,2w\})$ without changing that $\pi(U^*)=1$. Consequently, we can w.l.o.g assume that $X^-$ consists of the first $|X^-|$ elements from $\{1,3,\ldots,2w-1\}$ and that $X^+$ consists of the first $|X^+|$ elements from
$\{2,4,\ldots,2w\}$. But this means that $$T_\tau\mid U^*\Big(\{1\}\cup [t+2,t+2|X|]\Big).$$ As a result, setting $$J:=[t+1+\max\{2|X|,1\},t+2w]\subset J'\setminus\{1\}$$ and recalling from the beginning of the paragraph that $T_\alpha$ is disjoint from  $U^*(J')$, we find that  $T\mid \left[U^*\Big([1,|U|]\setminus J\Big)\right]$. It remains to estimate $|J|$.

Since $|X|=|T_\tau|\leq |T|\leq n^+-1$ holds by hypothesis, it follows in view of \eqref{n^+assump} that  \be\label{wickel}|J|=2w-\max\{2|X|,1\}=|U_\tau|-\max\{2|X|,1\}\geq |U_\tau|-\max\{2|T|,1\}\geq |U_\tau|-2n^++2.\ee
From \eqref{Vsmall} and \eqref{wickergoal}, we know $$|U_\tau|=|U|-|U_\alpha|\geq (n+n^++1)-(n^++n^--3)=n^+n^--n^-+4.$$ Combining this with \eqref{wickel} and making use of \eqref{n^+assump} and \eqref{n^-assump}, we find that
$$|J|\geq n^+n^--2n^+-n^-+6=3n^--6-n^-+6=2n^-,$$ completing the proof of Claim A.
\end{proof}

\bigskip

We will say that a subsequence $T\mid U$ is \emph{good} if it has an ordering $T^*\in \Fc^*(G)$, so  $[T^*]=T$, such that \be\label{Tnotation}T^*= y_1\bdot z_1\bdot\ldots\bdot  y_w\bdot z_w \bdot x_1\bdot\ldots\bdot  x_v\ee with $v,\,w\geq 0$,   $$x_i\in \la \alpha\ra\setminus \mathsf Z(G)\quad\mbox{ for $i\in [1,v]$,}\quad \und \quad y_i,\,z_i\in \tau\la \alpha\ra\quad\und \quad \iota(y_i)\not\equiv \iota(z_i)\mod p \quad\mbox{ for $i\in [1,w]$.}$$ Furthermore,
we define \be\label{starywavyness}\varphi(T^*)= (y_1z_1) \bdot\ldots\bdot
 (y_w z_w) \bdot x_1\bdot\ldots\bdot x_v\in \Fc^*(\la \alpha\ra)\quad\und\quad \ell(T)=|\varphi(T^*)|=v+w.\ee
We continue with the following claim.

\subsection*{Claim B} If $T\mid U$ is a good subsequence  with $\ell(T)\geq n^+-1=p-1$, then $\pi(T)$ is a $G'$-coset.

\begin{proof}
 Let $T^*\in \Fc^*(G)$ be an ordering from the definition of $T$ notated as in \eqref{Tnotation} and \eqref{starywavyness}.
Since  $\ell(T)\geq n^+-1=p-1$, it follows from \eqref{Vsmall} that $w\geq 1$. As remarked in Section \ref{sec-notation}, $\pi(T)$ is contained in a $G'$-coset. Therefore we need to show that $|\pi(T)|=|G'|=p$.

Since $w\geq 1$, it follows from \eqref{Rpi} and the definition of $\overline{g}$ that \be\label{tommow}|\pi(T)|=|\{0,\overline{x_1}\}+\ldots+\{0,\overline{x_v}\}+\Sigma_{w}
(\overline{T_\tau})|,\ee where $T_\tau\mid T$ is the subsequence of terms from $\tau\la \alpha\ra$. Note that $|T_\tau|=2w$.
Since $T$ is good, we know $x_i\in \la \alpha\ra\setminus \mathsf Z(G)=\la \alpha\ra\setminus \la \alpha^{p}\ra$ for $i\in [1,v]$, which means that $\overline{x_i}\neq 0$ for all $i\in [1,v]$.
Consequently,
since $p$ is prime, we can apply the Cauchy-Davenport Theorem to $\{0,\overline{x_1}\}+\ldots+\{0,\overline{x_v}\}$ to conclude \be\label{trepid}|\{0,\overline{x_1}\}+\ldots+\{0,\overline{x_v}\}|\geq \min\{p,\,v+1\}.\ee
Since $T$ is good, we have $\overline{y_i}\neq \overline{z_i}$ for $i\in [1,w]$, which together with the pigeonhole principle ensures that $\mathsf h(\overline{T_\tau})\leq w$. Consequently, since $p$ is prime, we can apply Theorem \ref{DGM-devos-etal-thm} to $\Sigma_{w}(\overline{T_\tau})$ to conclude
\be\label{trepidmore} |\Sigma_{w}(\overline{T_\tau})|\geq \min\{p,\,|T_\tau|-w+1\}=\min\{p,\,w+1\}.\ee Applying the Cauchy-Davenport Theorem to the $2$-fold sumset $\left(\{0,\overline{x_1}\}+\ldots+\{0,\overline{x_v}\}\right)+\Sigma_{w}(\overline{T_\tau})$,  using \eqref{trepid} and \eqref{trepidmore}, and recalling the case hypothesis $\ell(T)\geq p-1$, it follows that $$|\left(\{0,\overline{x_1}\}+\ldots+\{0,\overline{x_v}\}\right)+\Sigma_{w}(\overline{T_\tau})|\geq \min\{p,\,v+1+w+1-1\}=\min\{p,\,\ell(T)+1\}=p.$$ Combining this with \eqref{tommow} completes the proof of Claim B.
\end{proof}

\bigskip

Let $T\mid U$ be a good subsequence with $\ell(T)\geq 0$ maximal and  let $T^*\in \Fc^*(G)$ be an ordering from the definition of $T$ notated as in \eqref{Tnotation} and \eqref{starywavyness}. We handle two cases.

\subsection*{Case 1:} $\ell(T)\geq 2n^++n^--3$.

Recall the definition of $\varphi(T^*)$ given in \eqref{starywavyness}. We first proceed to show that there is a good subsequence $T'\mid T$ with \be\label{goodnesss} \ell(T')\geq n^+-1,\quad \pi(T')\subset G',\quad T\bdot {T'}^{[-1]}\quad\mbox{ a good subsequence,}\quad \und\quad \ell(T\bdot {T'}^{[-1]})\geq n^+-1.\ee
To do so, it suffices, in view of the case hypothesis $\ell(T)=|\varphi(T^*)|\geq 2n^++n^--3$,  to show that $[\phi_{G'}(\varphi(T^*))]$ has a product-one subsequence of length $\ell\in [n^+-1,n^+-2+n^-].$ Note that $$[\phi_{G'}(\varphi(T^*))]\in \Fc(\la \alpha\ra/G').$$
Thus, since $\ell(T)\geq n^+-2+n^-$ holds by hypothesis, and since $\mathsf d(\la \alpha\ra/G')+1\leq |\la \alpha\ra/G'|=n^-$ follows from Lemma \ref{dav-up-lemma}, such a subsequence can be found simply by repeated application of the definition of $\mathsf d(\la \alpha\ra/G')$ to $[\phi_{G'}(\varphi(T^*))]$.
This establishes \eqref{goodnesss}.

In view of \eqref{goodnesss} and Claim B, we have $\pi(T')=G'$. In particular, $T'$ is a nontrivial, product-one subsequence of $U$. Thus Lemma \ref{lemma-G'} shows that $\pi(U\bdot {T'}^{[-1]})\subset G'$.
As a result, since $T\bdot {T'}^{[-1]}\mid U\bdot {T'}^{[-1]}$, it follows in view of \eqref{goodnesss} and Claim B that $\pi(U\bdot {T'}^{[-1]})= G'$, so that $U\bdot{T'}^{[-1]}$ is also a product-one subsequence.
But now $U=T'\bdot(U\bdot{T'}^{[-1]})$ is a factorization of $U$ into $2$ nontrivial, product-one subsequences, contradicting that $U\in \mathcal A(G)$ is an atom. This completes Case 1.

\subsection*{Case 2:} $\ell(T)\leq 2n^++n^--4$.

In view of \eqref{Vsmall}, we know $|U_\alpha\bdot {U'_\alpha}^{[-1]}|\leq n^--1$.
We have  \be\label{gogetem} 2(2n^++n^--4)+1\leq |U|-|U_\alpha\bdot {U'_\alpha}^{[-1]}|,\ee for if \eqref{gogetem} failed, then $|U_\alpha\bdot {U'_\alpha}^{[-1]}|\leq n^--1$,
 \eqref{wickergoal}, \eqref{n^+assump} and \eqref{n^-assump}  would imply   $$0>|U|-4n^+-3n^-+8\geq n^+n^--3n^+-3n^-+9=(n^+-3)(n^--3)\geq 0,$$ which is a contradiction.
In view of the maximality of $\ell(T)$, we must have $U'_\alpha\mid T$.
Let $T_\tau=T\bdot {U'_\alpha}^{[-1]}$. Then, in view of the case hypothesis, it follows that  $T_\tau\mid U_\tau$ is a good subsequence with \be\label{W-max}\ell(T_\tau)=\ell(T)-|U'_\alpha|\leq   2n^++n^--|U'_\alpha|-4\quad\mbox{ maximal subject to $T_\tau\mid U_\tau$}.\ee

From \eqref{gogetem}, we deduce that \be\label{gabsam} 2(2n^++n^--4-|U'_\alpha|)+1\leq |U|-|U_\alpha\bdot {U'_\alpha}^{[-1]}|-|U'_\alpha|=|U_\tau|.\ee Now $|U_\tau|$ must be even as remarked in the paragraph above \eqref{relationtime}, which means that the inequality in \eqref{gabsam} must be strict: \be\label{ahahahahhah}2\ell:=2(2n^++n^--|U'_\alpha|-3)\leq |U_\tau|.\ee

It is readily seen that a subsequence $R\mid U_\tau$ being a good is equivalent to $\overline{R}$ having an $\frac12|R|$-setpartition with terms of as near equal a size as possible and $|R|$ even. In view of \eqref{W-max} and \eqref{Vsmall}, we see that $U_\tau$ does not have a good subsequence $R\mid U_\tau$ with $$\ell(R)=\frac{1}{2}|R|=\ell=2n^++n^--|U'_\alpha|-3\geq 1.$$ Thus applying Lemma \ref{lemma-setpartition-exsitence} to $\overline{U_\tau}$ taking $\ell=n$, we conclude that either $2\ell>|U_\tau|$ or there exists a nonempty subset $X\subset G$ with $|X|\leq \lfloor\frac{\ell-1}{\ell}+1\rfloor=1$ such that at least $|U_\tau|-\ell+1$ terms of $|\overline{U_\tau}|$ are all from $X$. In view of \eqref{ahahahahhah}, we see that the former is not possible, in which case the latter must hold, and with $|X|=1$. In other words, \be\label{miff}\mathsf h(\overline{U_\tau})\geq |U_\tau|-\ell+1.\ee

Now \eqref{miff} is equivalent to saying that there is some $x_0\in [0,p-1]$ such that all but at most $\ell-1$ terms of $U_\tau$ have the form $\tau\alpha^{x}$ with $x\equiv x_0\mod p$. However,
since $p=n^+=m^+\mid r+1$ follows from \eqref{tabel-n} in view of $\rho(P)=1$ and the definition of $m^+$, a short calculation shows that $$H:=\{\tau\alpha^x:\;x\in [0,n-1]\und x\equiv x_0\mod p\}\cup \{\alpha^y:\;y\equiv 0\mod p\}\leq G$$ is a subgroup of $G$ having $|H|=2n^-$. Indeed, $H=\mathsf C_G(\tau\alpha^{x_0})=\la \alpha^p,\tau\alpha^{x_0}\ra$, though we will not need this fact.

Let $U_H\mid U$ be the subsequence of $U$ with terms from $H$. In view of the two previous paragraphs, we see that \eqref{miff} is equivalent to saying  \be\label{little-exceptions}|U\bdot {U_H}^{[-1]}|\leq |U'_\alpha|+\ell-1=2n^++n^--4.\ee
As a result, we have \be\label{littlemore} |U_H|\geq 2n^-+n^+-1,\ee for if \eqref{littlemore} failed, then combining this with \eqref{little-exceptions} and \eqref{wickergoal} would yield
$$n^+n^-+n^++1\leq |U|=|U_H|+|U\bdot {U_H}^{[-1]}|\leq 3n^++3n^--6,$$ and then rearranging the above inequality and applying \eqref{n^+assump} and \eqref{n^-assump} yields the contradiction
$$0\geq n^+n^--2n^+-3n^-+7 \geq 3n^--6-3n^-+7=1.$$

Let $x$ be a term from $U\bdot U_H^{[-1]}$. If $x\in \la \alpha\ra$, then $x\in \supp(U'_\alpha)$ and $x$ can be included in a good sequence. On the other hand, if $x\in \tau\la \alpha\ra$, then $x$ can be paired with any term from  $U_H$ lying in $\tau \la\alpha \ra$ and thereby included in a good sequence. In particular, if we know that there are at least $t\geq 0$ terms of $U$ from $H\cap \tau\la \alpha\ra$, then $U$ possesses a good subsequence $R\mid U$ with $\ell(R)=\min\{t,|U\bdot U_H^{[-1]}|\}$.

\medskip

With these key facts finally established, we are now ready to finish the proof, which we do in $2$ short subcases.

\subsection*{Case 2.1} $|U\bdot {U_H}^{[-1]}|\geq n^+-1$.

Recall that $H\cap \la \alpha\ra=\la \alpha^p\ra=\mathsf Z(G)$ and that $U_\alpha\bdot {U'_\alpha}^{[-1]}$ is the subsequence of $U$ consisting of all terms from $\mathsf Z(G)$. Thus,
in view of \eqref{littlemore} and \eqref{Vsmall}, we can find a subsequence $U'_H\mid U_H$ with $U_\alpha \bdot {U'_\alpha}^{[-1]}\mid U'_H$ and $|U'_H|=2n^-$. Since  $U_\alpha\bdot {U'_\alpha}^{[-1]}$ is the subsequence of $U$ consisting of all terms from $\mathsf Z(G)=H\cap \la \alpha\ra$, since $U_\alpha \bdot {U'_\alpha}^{[-1]}\mid U'_H$ and since $\supp(U_H)\subset H$, we see that  \be\label{gobleek} \supp(U_H\bdot {U'_H}^{[-1]})\subset H\cap \tau\la \alpha\ra.\ee
Since $|U'_H|=2n^-=|H|$, applying Lemma \ref{dav-up-lemma} to $U'_H$ yields a nontrivial, product-one subsequence $R\mid U'_H$. From Lemma \ref{lemma-G'}, it follows that \be\label{woodoo}\pi(U\bdot R^{[-1]})\subset G'.\ee
Since $R\mid U'_H$, we have \be\label{dispelling} U\bdot {U'_H}^{[-1]}\mid U\bdot R^{[-1]}.\ee
Since $R\mid U'_H$ and $U'_H\mid U_H$, we have \be\label{dddispelling} U\bdot {U_H}^{[-1]}\mid U\bdot R^{[-1]}.\ee
From \eqref{littlemore}, we find that $$|U_H\bdot {U'_H}^{[-1]}|\geq 2n^-+n^+-1-|U'_H|= n^+-1.$$ Consequently, it follows in view of \eqref{gobleek} and \eqref{dispelling} that there are at least $n^+-1$ terms of $U\bdot R^{[-1]}$ from  $H\cap \tau\la \alpha\ra$. Combining this with \eqref{dddispelling} and applying the argument given just above Case 2.1, it follows that $U\bdot R^{[-1]}$ contains a good subsequence $T\mid U\bdot R^{[-1]}$ with $\ell(T)\geq \min\{n^+-1,|U\bdot {U_H}^{[-1]}|\}\geq n^+-1$, where the latter inequality follows in view of the subcase hypothesis. But now, in view of \eqref{woodoo}, we can apply Claim B to find that $\pi(U\bdot R^{[-1]})$ is not just contained in $G'$, but must be equal to $G'$, so $\pi(U\bdot R^{[-1]})=G'$. Hence $U=R\bdot (U\bdot R^{[-1]})$ is a factorization of $U$ into $2$ nontrivial, product-one subsequences, contradicting that $U\in\mathcal A(G)$ is an atom.

\subsection*{Case 2.2}$|U\bdot {U_H}^{[-1]}|\leq n^+-2$.

In this case, we can apply Claim A using $T=U\bdot {U_H}^{[-1]}$ to find an ordering of $U$, say $U^*\in \Fc^*(G)$ with $[U^*]=U$, and an interval $J\subset [1,|U|]$ such that $$\pi(U^*)=1,\quad U\bdot {U_H}^{[-1]}\mid U^*\Big([1,|U|]\setminus J\Big)\quad\und\quad |J|\geq |H|=2n^-.$$ In view of $U\bdot {U_H}^{[-1]}\mid U^*([1,|U|]\setminus J)$, we have $U^*(J)\mid U_H$. Thus $U^*(J)\in \Fc(H)$ with $|U^*(J)|=|J|\geq 2n^-=|H|$. As a result, applying Lemma \ref{dav-up-lemma} yields a nontrivial, consecutive, product-one subsequence $R^*$ in $U^*(J)$ with $|R^*|\leq 2n^-$. Since $U^*(J)\mid U^*$ is also consecutive (as $J\subset [1,|U|]$ is an interval), this means that $R^*\mid U^*$ is a nontrivial, consecutive, product-one sequence in $U^*$ with $U=[U^*]\in \mathcal A(G)$ an atom, in which case Lemma \ref{lemma-no-consecutive-prod-1} ensures that $R^*=U^*$. But then \eqref{wickergoal} and \eqref{n^+assump} give $$2n^-\geq |R^*|=|U^*|=|U|\geq n^+n^-+n^++1\geq 3n^-+4,$$ which is a proof concluding contradiction.
\end{proof}




\section*{Acknowledgement} We wish to thank the referee for their suggestions, particularly the idea leading to Theorem \ref{thm-induc-bound}, without which the paper would have remained much more limited in scope.


\bigskip

\begin{thebibliography}{10}

\bibitem{Ba-Ch11a}
P.~Baginski and S.T. Chapman, \emph{Factorizations of algebraic integers, block
  monoids, and additive number theory}, Am. Math. Mon. \textbf{118} (2011), 901
  -- 920.

\bibitem{Ba07b}
J.~Bass, \emph{Improving the {E}rd{\H{o}}s-{G}inzburg-{Z}iv theorem for some
  non-abelian groups}, J. Number Theory \textbf{126} (2007), 217 -- 236.

\bibitem{Be08a}
Y.~Berkovich, \emph{Groups of prime power order, vol {I}}, Expositions in
  Mathematics, vol.~46, de Gruyter, 2008.

\bibitem{B-D-G-L03}
A.~Bialostocki, P.~Dierker, D.~Grynkiewicz, and M.~Lotspeich, \emph{On some
  developments of the {E}rd{\H{o}}s-{G}inzburg-{Z}iv {T}heorem {II}}, Acta
  Arith. \textbf{110} (2003), 173 -- 184.

\bibitem{Ca96b}
Y.~Caro, \emph{Zero-sum problems - a survey}, Discrete Math. \textbf{152}
  (1996), 93 -- 113.

\bibitem{De-Go-Mo09a}
M.~DeVos, L.~Goddyn, and B.~Mohar, \emph{A generalization of {K}neser's
  addition theorem}, Adv. Math. \textbf{220} (2009), 1531 -- 1548.

\bibitem{Fa06a}
A.~Facchini, \emph{Krull monoids and their application in module theory},
  Algebras, {R}ings and their {R}epresentations (A.~Facchini, K.~Fuller, C.~M.
  Ringel, and C.~Santa-Clara, eds.), World Scientific, 2006, pp.~53 -- 71.

\bibitem{F-H-K-W06}
A.~Facchini, W.~Hassler, L.~Klingler, and R.~Wiegand, \emph{Direct-sum
  decompositions over one-dimensional {C}ohen-{M}acaulay local rings},
  Multiplicative {I}deal {T}heory in {C}ommutative {A}lgebra (J.W. Brewer,
  S.~Glaz, W.~Heinzer, and B.~Olberding, eds.), Springer, 2006, pp.~153 -- 168.

\bibitem{Ga-Ge06b}
W.~Gao and A.~Geroldinger, \emph{Zero-sum problems in finite abelian groups{\rm
  \,:} a survey}, Expo. Math. \textbf{24} (2006), 337 -- 369.

\bibitem{Ga-Li10b}
W.~Gao and Yuanlin Li, \emph{The {E}rd{\H{o}}s-{G}inzburg-{Z}iv theorem for
  finite solvable groups}, J. Pure Appl. Algebra \textbf{214} (2010), 898 --
  909.

\bibitem{Ga-Lu08a}
W.~Gao and Zaiping Lu, \emph{The {E}rd{\H{o}}s-{G}inzburg-{Z}iv theorem for
  dihedral groups}, J. Pure Appl. Algebra \textbf{212} (2008), 311 -- 319.

\bibitem{Ge09a}
A.~Geroldinger, \emph{Additive group theory and non-unique factorizations},
  Combinatorial {N}umber {T}heory and {A}dditive {G}roup {T}heory
  (A.~Geroldinger and I.~Ruzsa, eds.), Advanced Courses in Mathematics CRM
  Barcelona, Birkh{\"a}user, 2009, pp.~1 -- 86.

\bibitem{Ge-HK06b}
A.~Geroldinger and F.~Halter-Koch, \emph{Non-unique factorizations{\rm \,:} a
  survey}, Multiplicative {I}deal {T}heory in {C}ommutative {A}lgebra (J.W.
  Brewer, S.~Glaz, W.~Heinzer, and B.~Olberding, eds.), Springer, 2006, pp.~207
  -- 226.

\bibitem{Ge-HK06a}
\bysame, \emph{Non-{U}nique {F}actorizations. {A}lgebraic, {C}ombinatorial and
  {A}nalytic {T}heory}, Pure and Applied Mathematics, vol. 278, Chapman \&
  Hall/CRC, 2006.

\bibitem{Ge-Ne13b}
A.~Geroldinger and M.D. Neusel, \emph{On the interplay between invariant theory
  and zero-sum theory}, manuscript.

\bibitem{Gr14a}
D.J. Grynkiewicz, \emph{The large {D}avenport constant {II:} {G}eneral upper
  bounds}, manuscript.

\bibitem{Gr13a}
\bysame, \emph{Structural {A}dditive {T}heory}, to appear, 2013.

\bibitem{HK98}
F.~Halter-Koch, \emph{Ideal {S}ystems. {A}n {I}ntroduction to {M}ultiplicative
  {I}deal {T}heory}, Marcel Dekker, 1998.

\bibitem{HK08a}
\bysame, \emph{Non-unique factorizations of algebraic integers}, Funct.
  Approximatio, Comment. Math. \textbf{39} (2008), 49 -- 60.

\bibitem{Jo90a}
D.L. Johnson, \emph{Presentations of {G}roups}, Student Texts, vol.~15, London
  Math. Soc., 1990.

\bibitem{Jo-Kw-Xu13a}
G.A. Jones, Jinho Kwak, and Mingyao Xu, \emph{Finite {G}roup {T}heory for
  {C}ombinatorists}, Chapman \& Hall/CRC, to appear.

\bibitem{Na79}
W.~Narkiewicz, \emph{Finite abelian groups and factorization problems}, Colloq.
  Math. \textbf{42} (1979), 319 -- 330.

\bibitem{Ne07b}
M.~D. Neusel, \emph{Degree bounds -- an invitation to postmodern invariant
  theory}, Topology Appl. \textbf{154} (2007), 792 -- 814.

\bibitem{Ne07a}
\bysame, \emph{{I}nvariant {T}heory}, Student {M}athematical {L}ibrary,
  vol.~36, AMS, 2007.

\bibitem{Ne-Sm02a}
M.~D. Neusel and L.~Smith, \emph{{I}nvariant {T}heory of {F}inite {G}roups},
  {M}athematical {S}urveys and {M}onographs, vol.~94, AMS, 2002.

\bibitem{Ol-Wh77}
J.E. Olson and E.T. White, \emph{Sums from a sequence of group elements},
  Number {T}heory and {A}lgebra (H.~Zassenhaus, ed.), Academic {P}ress, 1977,
  pp.~215 -- 222.

\bibitem{Ro63}
K.~Rogers, \emph{A combinatorial problem in abelian groups}, Proc. Camb.
  Philos. Soc. \textbf{59} (1963), 559 -- 562.

\bibitem{Sc09b}
W.A. Schmid, \emph{Characterization of class groups of {K}rull monoids via
  their systems of sets of lengths{\rm \,:} a status report}, Number {T}heory
  and {A}pplications{\rm \,:} {P}roceedings of the {I}nternational
  {C}onferences on {N}umber {T}heory and {C}ryptography (S.D. Adhikari and
  B.~Ramakrishnan, eds.), Hindustan Book Agency, 2009, pp.~189 -- 212.

\bibitem{Wa-Ga08a}
Guoquing Wang and Weidong Gao, \emph{Davenport constants for semigroups},
  Semigroup Forum \textbf{76} (2008), 234 -- 238.

\bibitem{Za56a}
H.~Zassenhaus, \emph{The theory of groups, 2nd ed.}, Vandenhoeck \& Ruprecht,
  G{\"o}ttingen, 1956.

\bibitem{Zh-Ga05a}
J.~Zhuang and W.~Gao, \emph{{E}rd{\H{o}}s-{G}inzburg-{Z}iv theorem for dihedral
  groups of large prime index}, Eur. J. Comb. \textbf{26} (2005), 1053 -- 1059.

\end{thebibliography}

\providecommand{\bysame}{\leavevmode\hbox to3em{\hrulefill}\thinspace}
\providecommand{\MR}{\relax\ifhmode\unskip\space\fi MR }
\providecommand{\MRhref}[2]{%
  \href{http://www.ams.org/mathscinet-getitem?mr=#1}{#2}
}
\providecommand{\href}[2]{#2}

\end{document}